\documentclass[a4paper,10pt]{amsart}

\usepackage{amssymb,amsmath,amsthm,mathrsfs,enumerate,graphicx, color}
\usepackage[pdfpagelabels,colorlinks,linkcolor=blue,citecolor=black,urlcolor=blue]{hyperref}
\usepackage{esint}
\newtheorem{thm}{Theorem}[section]

\newtheorem{lem}[thm]{Lemma}
\newtheorem{prop}[thm]{Proposition}
\newtheorem{defn}[thm]{Definition}
\newtheorem{rem}[thm]{Remark}

\newcommand{\lesi}{\lesssim}

\newcommand{\f}{\frac}

\newcommand{\Om}{\Omega}

\newcommand{\vc}{\infty}
\newcommand{\Rn}{\mathbb{R}^n}

\newcommand{\di}{{\rm div}}

\title[Global Marcinkiewicz estimates for parabolic equations]{Global Marcinkiewicz estimates for nonlinear parabolic equations with nonsmooth coefficients}         
\author{The Anh Bui}
\address{Department of Mathematics, Macquarie University, NSW 2109,
Australia}
\email{the.bui@mq.edu.au, bt\_anh80@yahoo.com}
 
\author{Xuan Thinh Duong}
\address{Department of Mathematics, Macquarie University, NSW 2109,
Australia}

\email{xuan.duong@mq.edu.au}


\begin{document}

\date{}

\maketitle

\begin{abstract} 
	Consider the parabolic equation with measure data
	\begin{equation*}
	\left\{
	\begin{aligned}
	&u_t-\di \mathbf{a}(D u,x,t)=\mu&\text{in}& \quad \Om_T,\\
	&u=0 \quad &\text{on}& \quad \partial_p\Om_T,
	\end{aligned}\right.
	\end{equation*}
	where $\Om$ is a bounded domain in $\Rn$, $\Om_T=\Om\times (0,T)$,  $\partial_p\Om_T=(\partial\Om\times (0,T))\cup (\Om\times\{0\})$, and $\mu$ is a signed Borel measure with finite total mass. Assume that the nonlinearity ${\bf a}$ satisfies a small BMO-seminorm condition, and $\Om$ is a Reifenberg flat domain. This paper proves a  global Marcinkiewicz estimate for the SOLA (Solution Obtained as Limits of Approximation) to the parabolic equation. 
\end{abstract}

\tableofcontents

\section{Introduction}

Let $\Om$ be a bounded open domain in $\Rn$, $n\geq 2$. For $p\geq 2$, we consider the following parabolic equation with measure data
\begin{equation}\label{ParabolicProblem}
\left\{
\begin{aligned}
&u_t-\di \mathbf{a}(D u,x,t)=\mu&\text{in}& \quad \Om_T,\\
&u=0 \quad &\text{on}& \quad \partial_p\Om_T,
\end{aligned}\right.
\end{equation}
where $T>0$ is a given positive constant, $\Om_T=\Om\times (0,T)$,  $\partial_p\Om_T=(\partial\Om\times (0,T))\cup (\bar \Om\times\{0\})$, and $\mu$ is a signed Borel measure with finite total mass.  Throughout the paper, we denote $u_t=\f{\partial u}{\partial t}$ and $Du=D_xu:=(D_{x_1},\ldots, D_{x_n})$.

In this paper, we assume that the nonlinearity $\mathbf{a}(\xi,x,t)=(\mathbf{a}^1,\ldots,\mathbf{a}^n): \mathbb{R}^n\times \mathbb{R}^n\times \mathbb{R}\to \mathbb{R}^n$ in \eqref{ParabolicProblem} is measurable in $(x,t)$ for every $\xi$, differentiable in $\xi$ for a.e. $(x,t)$, and  satisfies the following conditions: there exist $\Lambda_1, \Lambda_2>0$ so that 
	\begin{equation}\label{eq1-functiona}
	|\mathbf{a}(\xi,x,t)|+|\xi||D_\xi\mathbf{a}(\xi,x,t)|\leq \Lambda_1|\xi|^{p-1},
	\end{equation}
	and
	\begin{equation}\label{eq2-functiona}
	\langle \mathbf{a}(\xi,x,t)-\mathbf{a}(\eta,x,t), \xi -\eta\rangle\geq \Lambda_2
	|\xi-\eta|^p
	\end{equation}
	for a.e $(\xi,\eta) \in \mathbb{R}^n\times \mathbb{R}^n$ and a.e. $(x,t)\in \mathbb{R}^n\times \mathbb{R}$.

Note that a standard example of such a nonlinearity $\mathbf{a}(\xi,x,t)$ satisfying these conditions is the
$p$-Laplacian $\Delta_p u={\rm div}(|D u|^{p-2}Du)$ with respect to $\mathbf{a}(\xi,x,t)=|\xi|^{p-2}\xi$. This general nonlinearity was studied for both elliptic and parabolic equation by many authors. See for example \cite{AM1, KL, Gia, GM, IS, JK, LSU, BW2, BDL, BD} and the reference therein.

\begin{defn}
A function $u\in C(0,T; L^2(\Om))\cap L^p(0, T; W^{1,p}_0(\Om))$ is said to be a weak solution to the equation \eqref{ParabolicProblem} if the following holds true
\begin{equation}
\label{eq-weak solution}
-\int_{\Om_T}u\varphi_t dxdt +\int_{\Om_T}\langle {\bf a}(Du, x,t), D\varphi\rangle dxdt = \int_{\Om_T}\varphi d\mu,
\end{equation} 
for every test function $\varphi \in C^\vc(\Om_T)$ vanishing in a neighborhood of $\partial_p\Om_T$.
\end{defn}
\begin{rem}
	Due to the lack of regularity with respect to the time variable, the weak solution $u$ to the problem \eqref{ParabolicProblem} could not be choosen as a test function in the formula \eqref{eq-weak solution}. In order to overcome this trouble, we make use of the Steklov averages or the standard mollifiers. For further details, we refer to, for example, \cite{B, P2}.
\end{rem}

In general, it is not clear whether the weak solution to the equation \eqref{ParabolicProblem} exists. For this reason, the notion of SOLA (Solution Obtained as Limits of Approximation) will be employed in this situation. For the sake of convenience, we sketch the ideas of an approximation scheme in \cite{BG, BG2,B.etal}. For each $k\in \mathbb{N}$, we consider the regularized problem
  \begin{equation}\label{RegularizedParabolicProblem}
  \left\{
  \begin{aligned}
  &(u_k)_t-\di \mathbf{a}(D u_k,x,t)=\mu_k&\text{in}& \quad \Om_T,\\
  &u_k=0 \quad &\text{on}& \quad \partial_p\Om_T,
  \end{aligned}\right.
  \end{equation}
  where $\mu_k\in C^\vc(\Om_T)$ converges to $\mu$ in the weak sense of measure and
  $$
  |\mu_k|(Q_R\cap \Om_T)\leq |\mu|(Q_R\cap \Om_T), \ \ \ k\geq 1, R>0.
  $$
  As a classical result, the equation \eqref{RegularizedParabolicProblem} admits a weak solution $u_k\in C(0,T; L^2(\Om))\cap L^p(0, T; W^{1,p}_0(\Om))$ for each $k$. Moreover, it was proved in \cite{B.etal} that there exists $u$ so that $u_k \to u$ in $L^{q}(0, T; W^{1,q}_0(\Om))$ for any $q\in [1,p-1+\f{1}{n+1})$. By this reason, the limit of approximation solution $u$ is refered to SOLA (Solution Obtained as Limits of Approximation). In the general case, the  SOLA may not be unique. However, in our situation the uniqueness of SOLA is guaranteed by $\mu\in L^1(\Om_T)$. See for example \cite{Da}.
 
  Let $0<\theta\leq n+2$, we say that the measure $\mu$ is in the Morrey space $ L^{1,\theta}(\Om_T)$ if the following holds true:
  $$
  \sup_{z\in \Om_T}\sup_{0<r\leq {\rm diam} \Om_T}\f{|\mu|(Q_r(z)\cap \Om_T)}{|Q_r(z)\cap \Om_T|^{1-\f{\theta}{n+2}}}<\vc,
  $$
  where $Q_r(z)=B_r(x)\times (t-r^2,t+r^2)$ with $z=(x,t)$ and $B_r(x)=\{y\in \Rn: |x-y|<r\}$.
   
 The nonlinear elliptic and parabolic equations with measure data have received a great deal of attention by many mathematicians. See for example \cite{Be.etal, B.etal, BG, BG2, DM1, DM2, KM, MV, Min1, Min2, Min3} and the references therein. One of the most interesting problems concerning the SOLAs to the equation \eqref{ParabolicProblem} is the Marcinkiewicz type estimate. More precisely, we look for suitable conditions on the nonlinearity ${\bf a}$ and the domain $\Om$ so that the following implication holds true
    \begin{equation}\label{eq-Mac Estimates}
    \mu\in L^{1,\theta}(\Om_T), \ \ \theta\in (1,n+2] \ \ \ \Longrightarrow |Du|\in\mathcal{M}^{m}(\Om_T)
    \end{equation}
    for some $m=m(p,\theta)$,  where $\mathcal{M}^{m}(\Om_T)$ is the weak-Lebesgue space, or the Marcinkiewicz space, defined by the set of all measurable functions $f$ on $\Om_T$ satisfying
 $$
 \|f\|_{\mathcal{M}^{m}(\Om_T)}:=\sup_{\lambda>0} \lambda|\{z\in \Om_T: |f(z)|>\lambda\}|^{\f{1}{m}}<+\vc.
 $$
 The usual modification is used to define the Marcinkiewicz space on any measurable subset $E\subset \Om_T$.

 In \cite{Min1}, the local Marcinkiewicz type estimates \eqref{eq-Mac Estimates} were obtained for the elliptic equations with Morrey data:
 $$
 \mu\in L^{1,\theta}(\Om), 2\leq \theta\leq n \Longrightarrow |Du|\in \mathcal{M}_{\rm loc}^{\f{\theta(p-1)}{\theta-1}}(\Om).
 $$
 Note that when $\theta=n$, the above estimate reads 
  $$
  \mu\in L^{1,n}(\Om) \Longrightarrow |Du|^{p-1}\in \mathcal{M}_{\rm loc}^{\f{n}{n-1}}(\Om),
  $$  
 which was proved in \cite{B.etal, BG} for $p<n$. The borderline case $p=n$ is much more difficult and was investigated in \cite{DHM}.
 
 For the parabolic equation, the local version of Marcinkiewicz type estimates \eqref{eq-Mac Estimates} for $p=2$ was obtained in \cite{BH} by making use of the maximal function technique. The case $p\geq 2$ is more complicated and has been studied recently in \cite{Ba}. More precisely, the author in  \cite{Ba} proved that there exists $\tilde{\theta}\in (1,2)$ so that 
 $$
 \mu\in L^{1,\theta}(\Om_T), \ \ \theta\in (\tilde{\theta},n+2] \ \ \ \Longrightarrow |Du|\in\mathcal{M}_{\rm loc}^{m}(\Om_T), \ \ m=p-1+\f{1}{\theta-1}.
 $$
 The number $\tilde{\theta}\in (1,2)$ is a threshold and has a connection with the exponent in higher integrability estimates of the associated homogeneous equation. It is also claimed in \cite{Ba} that the range $\theta\in (\tilde{\theta},n+2]$ can be improved to be $\theta\in (1,n+2]$ if either ${\bf a}(\xi,x,t)=b(x)a(\xi,t)$ and $b(\cdot)$ satisfies certain VMO regularity conditions, or ${\bf a}(\xi,x,t)$ is continuous with respect to $x$ with some additional smoothness conditions.

This paper is devoted to the global Marcinkiewicz estimates \eqref{eq-Mac Estimates} with the general class of nonlinearities ${\bf a}$ and the non-smooth domains. Our main result is the following theorem.
\begin{thm}\label{mainthm1}
	For any $1<\theta\leq n+2$, there
	exists a positive constant $\delta$ such that the following holds. If $\mu\in L^{1,\theta}(\Om_T)$, the domain $\Omega$ is a $(\delta, R_0)$-Reifenberg flat domain (see Definition \ref{defn2}), and the nonlinearity ${\bf a}$  satisfies \eqref{eq1-functiona}, \eqref{eq2-functiona} and the small $(\delta, R_0)$-BMO condition \eqref{eq2-Assumption} (see Definition \ref{defn-smallBMO} for \eqref{eq2-Assumption}), then the problem \eqref{ParabolicProblem} has a unique SOLA $u$ such that 
	\begin{equation}\label{eq-mainthm}
	\||Du|\|_{\mathcal{M}^m(\Om_T)}\leq C\Big[|\mu|(\Om_T)^{\f{n}{n+1}}+1\Big], \ \  m=p-1+\f{1}{\theta-1}, 
	\end{equation}
	where $C$ is a constant depending on $n, \Lambda_1,\Lambda_2, \delta, R_0$ and $\Om_T$.
\end{thm} 
\begin{rem}
	\noindent (a) In Theorem \ref{mainthm1}, we are only interested in $\theta\in (1,n+2]$. The case $\theta\in (0,1]$ can be deduced immediately from the estimate for $\theta\in (1,n+2]$. Indeed, if $\mu \in L^{1,\tilde{\theta}}(\Om_T)$ for some $\tilde{\theta}\in (0,1]$, then from the definition we have $L^{1,\theta}(\Om_T)$ for any $\theta\in (1,n+2]$. Applying Theorem \ref{mainthm1} and letting $\theta\to 1^+$, we obtain $|Du|\in \mathcal{M}^q(\Om_T)$ for any $p-1<q<\vc$. Hence, $|Du|\in L^q(\Om_T)$ for any $p-1<q<\vc$.\\
	
	\noindent (b) It is not clear whether the exponent $\f{n}{n+1}$ on the right hand side of \eqref{eq-mainthm} is optimal. This problem is, of course, interesting in its own right, but we do not pursue it in this paper. 
\end{rem}

It is important to stress that although the local Marcinkiewicz estimates have been investigated intensively for elliptic and parabolic equations, (see for example  \cite{Min1, Ba} and the references therein), the global Marcinkiewicz estimates have not been obtained. Hence, the result in Theorem \ref{mainthm1} gives a new result on the global  Marcinkiewicz estimate for nonlinear parabolic equations with measure data. We note that in Theorem \ref{mainthm1}, we require neither continuity conditions of the nonlinearity ${\bf a}$, nor smoothness conditions on the boundary $\partial \Om$. See Section 2 for further discussion on these two conditions.\\

We now give some comments on the technique used in this paper. In the particular case $p=2$, the Marcinkiewicz estimate can be otained by using maximal function techniques. See for example \cite{BH}. However, this harmonic analysis tool does not work well for the case $p\neq 2$, mainly because the homogeneity of the parabolic equations is no longer true as $p\neq 2$, even when $\mu\equiv 0$. To overcome this trouble, we adapt the technique introduced in  \cite{AM, AM1} which makes use of the approximation method in \cite{CP} and the Vitali covering lemma. This method is an effective tool in studying the general nonlinear parabolic equations. See for example \cite{AM1, AM, Ba, BR1}. \\

\medskip

The organization of the paper is as follows. In Section 2, we give the assumptions used in the paper. Some important approxiation results for the solution to the problem \eqref{ParabolicProblem} are represented in Section 3. The proof of Theorem \ref{mainthm1} is represented in Section 4.
\\

Throughout the paper, we always use $C$ and $c$ to denote positive constants that are independent of the main parameters involved but whose values may differ from line to line. We will write
$A\lesi B$ if there is a universal constant $C$ so that $A\leq CB$ and $A\sim B$ if $A\lesi B$ and $B\lesi A$. We denote by $\mathcal{O}(\texttt{data})$ the small quantity such that $\lim_{\texttt{data}\to 0}\mathcal{O}(\texttt{data})=0$.

\section{Our assumptions}
For $r,\tau, \lambda>0$, $z=(x,t)$ with $x\in \mathbb{R}^n, t>0$, we first introduce some notations which will be used in the paper:
\begin{itemize}
	\item $\Om_T=\Om\times (0,T)$ and $\partial_p \Om_T=(\partial \Om\times [0, T])\cup (\bar \Om\times\{0\})$.
	\item $B_r=\{y: |y|<r\}$, $\Om_r=B_r\cap \Om$, $B_r^+=B_r\cap \{y=(y_1,\ldots,y_n): y_n>0\}$, and $B_r(x)=x+B_r$, $\Om_r(x)=x+\Om_r$, $B^+_r(x)=x+B^+_r$.
	\item $Q_{r,\tau}=B_r\times (-\tau, \tau)$, $Q_{r,\tau}(z)=z+Q_{r,\tau}$, $K_{r,\tau}(z)=Q_{r,\tau}(z)\cap \Om_T$.
	\item $Q_{r}=Q_{r,r^2}$, $Q_{r}^+=Q_{r}\cap\{z=(x',x_n,t): x_n>0\}$, and  $Q_{r}(z)=z+Q_{r}$, $Q^+_{r}(z)=z+Q^+_{r}$.
	\item $\partial_p Q_r=\partial Q_r\backslash (\bar B_r\times \{r^2\})$, $\partial_p Q_r(z)=z+\partial_p Q_r$
	\item $K_{r}(z)=Q_{r}(z)\cap \Om_T$, $\partial_w K_r(z)=Q_r(z)\cap (\partial \Om\times \mathbb{R})$, $\partial_p K_r(z)=\partial K_r(z)\backslash (\Om_r(x)\times \{t+r^2\})$. 
	\item $I^{\lambda}_{r}(t)=(t-\lambda^{2-p}r^2, t+\lambda^{2-p}r^2)$, $Q_r^\lambda(z)=B_r(x)\times I^{\lambda}_{r}(t)$, $\partial_p Q^\lambda_r(z)=\partial Q^\lambda_r(z)\backslash (\bar B_r(x)\times \{t+\lambda^{2-p}r^2\})$.
	\item $K_{r}^\lambda(z)=Q^\lambda_{r}(z)\cap \Om_T$, $\partial_w K_r^\lambda(z)=Q_r^\lambda(z)\cap (\partial \Om\times \mathbb{R})$, $\partial_p K^\lambda_r(z)=\partial K^\lambda_r(z)\backslash (\bar\Om_r(x)\times \{t+\lambda^{2-p}r^2\})$.	
	\item \noindent For a measurable function $f$ on a measurable subset $E$ in $\mathbb{R}^{n}$ (or in $\mathbb{R}^{n+1}$) we define
	$$
	\overline{f}_E =\fint_E f =\f{1}{|E|}\int_E f.
	$$
\end{itemize}

\subsection{The small {\rm BMO}-seminorm condition}
Assume that the nonlinearity ${\bf a}$ satisfy \eqref{eq1-functiona} and \eqref{eq2-functiona}. We set
$$
\Theta(\mathbf{a},B_r(y))(x,t)=\sup_{\xi\in \mathbb{R}^n\backslash \{0\}}\f{|\mathbf{a}(\xi,x,t)-\overline{\mathbf{a}}_{B_r(y)}(\xi,t)|}{|\xi|^{p-1}}
$$
where
$$
\overline{\mathbf{a}}_{B_r(y)}(\xi,t)=\fint_{B_r(y)}\mathbf{a}(\xi,x,t)dx.
$$
\begin{defn}\label{defn-smallBMO}
	Let $R_0,\delta>0$. The nonlinearity ${\bf a}$ is said to satisfy the small $(\delta, R_0)$-BMO condition if 
	\begin{equation}\label{eq2-Assumption}
	[{\bf a}]_{2,R_0}:=\sup_{y\in \mathbb{R}^n}\sup_{0<r\leq R_0, 0<\tau<r^2}\,\fint_{Q_{(r,\tau)}(y)}|\Theta(\mathbf{a},B_r(y))(x,t)|^2dxdt\leq \delta^2.
	\end{equation}
\end{defn}
\begin{rem}
	\noindent (a) The nonlinearity ${\bf a}$ satisfying the small $(\delta, R_0)$-BMO condition \eqref{eq2-Assumption} is assumed to be merely measurable only in the time variable $t$ and belong to the  BMO class (functions with bounded mean oscillations) as functions of the spatial variable $x$. To see this, we consider the following example. If  ${\bf a}(\xi,x,t)=b(\xi,x)c(t)$, then \eqref{eq2-Assumption} requires small BMO norm regularity for $b(\xi,\cdot)$, whereas $c(\cdot)$ is just needed to be bounded and measurable. This is weaker than those used in \cite{BR1, BW2} in which the nonlinearity ${\bf a}$ is required to belong to the  BMO class in both variables $t$ and $x$. Note that the condition \eqref{eq2-Assumption} is similar to that used in \cite{K} to study the parabolic and elliptic equations with VMO coefficients. We refer to \cite{Sa} for the definition of VMO functions.

	\noindent (b) Under the conditions \eqref{eq1-functiona}, \eqref{eq2-functiona} and the small $(\delta, R_0)$-BMO condition \eqref{eq2-Assumption}, it is easy to see that for any $\gamma\in [1,\vc)$ there exists $\epsilon>0$ so that 
	$$
	[{\bf a}]_{\gamma,R_0}
	:=\sup_{y\in \mathbb{R}^n}\sup_{0<r\leq R_0, 0<\tau<r^2}\,\fint_{Q_{(r,\tau)}(y)}|\Theta(\mathbf{a},B_r(y))(x,t)|^\gamma dxdt\lesi \delta^\epsilon.
	$$
\end{rem}
\subsection{Reifenberg flat domains}
Concerning  the underlying domain $\Omega$, we do not assume any smoothness condition on $\Omega$, but the following  flatness condition.
\begin{defn}\label{defn2}
	Let $\delta, R_0>0$. The domain $\Om$ is said to be a $(\delta, R_0)$ Reifenberg flat domain if for every $x\in \partial\Om$ and $0<r\leq R_0$, then there exists a coordinate system depending on $x$ and $r$, whose variables are denoted by $y=(y_1,\dots,y_n)$ such that in this new coordinate system $x$ is the origin and
	\begin{equation}\label{eq1-Assumption}
	B_{r}\cap\{y: y_n>\delta r\}\subset B_{r}\cap \Omega \subset \{y: y_n>-\delta r\}.
	\end{equation}
\end{defn}

\begin{rem}\label{rem1}
	(a) The condition of  $(\delta, R_0)$-Reifenberg flatness condition was first introduced in \cite{R}. This condition does not require any smoothness on the boundary of $\Om$, but
	sufficiently flat in the Reifenberg's sense. The Reifenberg flat domain includes domains with rough boundaries of fractal
	nature, and Lipschitz domains with small Lipschitz constants.
	For further discussions about the Reifenberg domain, we refer to \cite{R, DT, Toro, P} and the references therein. 
	
	(b) If $\Om$ is a $(\delta, R_0)$ Reifenberg domain, then for any $x_0\in \partial \Om$ and $0<\rho<R_0(1-\delta)$ there exists a coordinate system, whose variables are denoted by $y=(y_1,\ldots, y_n)$ such that in this coordinate system the origin is some interior point of $\Om$, $x_0=(0,\ldots, 0, -\f{\delta \rho}{1-\delta})$ and
	$$
	B_{\rho}^+\subset B_{\rho}\cap \Om\subset B_{\rho}\cap  \left\{y: y_n>-\f{2\delta \rho}{1-\delta}\right\}.
	$$

	(c) For $x\in \Om$ and $0<r<R_0$, we have
	\begin{equation}
	\label{eq1-Reifenberg domain}
	\f{|B_r(x)|}{|B_r(x)\cap \Om|}\leq \Big(\f{2}{1-\delta}\Big)^n.
	\end{equation}
\end{rem}

\bigskip

{\it Throughout the paper, we always assume that the domain $\Om$ is a $(\delta, R_0)$ Reifenberg flat domain, and the nonlinearity ${\bf a}$ satisfies \eqref{eq1-functiona}, \eqref{eq2-functiona} and the small $(\delta, R_0)$-BMO condition \eqref{eq2-Assumption}.}
\subsection{Sobolev-Poincar\'e inequality on Reifenberg domains}

Let $1<p<\vc$ and $E$ be a compact subset in $\Om$. The $p$-capacity of a compact set $E$ which is denoted by $C_p(E,\Om)$ is defined by
$$
C_p(E,\Om)=\inf\left\{\int_{\Om}|Dg|^pdx: g\in C^\vc_0(\Om), g=1 \ \text{in} \ \ E\right\}.
$$  
It is well known that for $1<p<\vc$ and $r>0$,
\begin{equation}\label{cap ball}
C_p(\overline{B}_r, B_{2r})=cr^{n-p},
\end{equation}
where $c$ depends on $n$ and $p$. See for example \cite{KK,Ma}.

\begin{lem}
	\label{SP inequality}
	Suppose that $1<q<\vc$ and that $u$ is a $q$-quasicontinuous function in $W^{1,q}(B)$, where $B$ is a ball. Let $N_B(u)=\{x\in B: u(x)=0\}$. Then
	$$
	\Big(\fint_B|u|^{\kappa q}dx\Big)^{\f{1}{\kappa q}}\leq c\Big(\f{1}{C_q(N_B(u),2B)}\int_B|\nabla u|^qdx\Big)^{1/q},
	$$
	where $c = c(n, q) > 0$ and $\kappa = n/(n - q)$ if $1 <q < n$ and $\kappa = 2$ if $q\geq n$.
\end{lem}

In the particular case when $\Om$ is a Reifenberg flat domain, we have the following result.
\begin{lem}
	\label{SP inequality}
	Let $\Om$ is a $(\delta, R_0)$ Reifenberg domain. Suppose that $1<q<\vc$ and that $u$ is a $q$-quasicontinuous function in $W^{1,q}(\Om_r(x_0))$, where $x_0\in \partial\Om$ and $0<r<R_0$. Then
	\begin{equation}\label{SP-inequality 1}
	\Big(\fint_{\Om_r(x_0)}|u|^{\kappa q}dx\Big)^{\f{1}{\kappa q}}\leq cr\Big(\fint_{B_r(x_0)}|\nabla \bar{u}|^qdx\Big)^{1/q},
	\end{equation}
	where $c = c(n, q) > 0$ and $\kappa = n/(n - q)$ if $1 <q < n$ and $\kappa = 2$ if $q\geq n$, and $\bar{u}$ is the zero extension of $u$ from $\Om_r(x_0)$ to $B_r(x_0)$.
	
	In particularly, we have
	\begin{equation}\label{SP-inequality 2}
	\Big(\fint_{\Om_r(x_0)}|u|^{q}dx\Big)^{\f{1}{q}}\leq cr\Big(\fint_{B_r(x_0)}|\nabla \bar{u}|^qdx\Big)^{1/q}.
	\end{equation}
\end{lem}

\begin{proof}
	The inequality \eqref{SP-inequality 1} follows immediately from the definition of a $(\delta, R_0)$ Reifenberg domain, \eqref{cap ball} and Lemma \ref{SP inequality}.  The inequality \eqref{SP-inequality 2} follows from  \eqref{SP-inequality 1} and H\"older's inequality.
\end{proof}

\section{Interior estimates}
For $z_0=(x_0,t_0)\in \Om_T$, $0<R<R_0/4$ and $\lambda\geq 1$ satisfying $B_{4R}\equiv B_{4R}(x_0)\subset \Om$, we set 
\begin{equation}\label{defn-QR}
Q^\lambda_{4R}\equiv Q_{4R}^\lambda(z_0)=B_{4R}(x_0)\times I^{\lambda}_{4R}(t_0).
\end{equation}

For the sake of simplicity, we may assume that $I^{\lambda}_{4R}(t_0)\subset (0,T)$, or equivalently, $Q^\lambda_{4R}\subset \Om_T$. The case $I^{\lambda}_{4R}(t_0)\cap (0,T)^c\neq \emptyset$ can be done in the same manner with minor modifications.\\

Assume that $u$ is a weak solution to \eqref{ParabolicProblem}. It is well-known that there exists a unique weak solution $w \in C(I^\lambda_{4R}(t_0); L^2(B_{4R}(x_0)))\cap L^p(I^\lambda_{4R}(t_0); W^{1,p}(B_{4R}(x_0)))$ to the following equation
\begin{equation}\label{AppProb1-interior}
\left\{
\begin{aligned}
&w_t-\di\, \mathbf{a}(D w,x,t)=0 \quad &\text{in}& \quad Q^\lambda_{4R},\\
&w=u \quad &\text{on}& \quad \partial_p Q^\lambda_{4R}.
\end{aligned}\right.
\end{equation}

Then we have the following estimate. See Lemma 4.1 in \cite{KM}.
\begin{lem}\label{lem1-inter}
	Let $w$ be a weak solution to the problem \eqref{AppProb1-interior}. Then for every $1\leq q<p-1+\f{1}{n+1}$, there exists $C$ so that
	\begin{equation}
	\label{eq1-lem1-inter}
	\Big(\fint_{Q^\lambda_{4R}}|D(u-w)|^qdxdt\Big)^{1/q}\leq C \left[\f{|\mu|(Q^\lambda_{4R})}{|Q^\lambda_{4R}|^{(n+1)/(n+2)}}\right]^{\f{n+2}{p+(p-1)n}}.
	\end{equation}
\end{lem}

Moreover, we have the following higher integrability property. 
\begin{prop}
	\label{higherInte-Prop1-inter}
	Let $w$ be a weak solution to the problem \eqref{AppProb1-interior}. Assume that 
	\begin{equation}\label{eq- Dw condition}
	\kappa^{-1}\lambda^p\leq \fint_{Q^\lambda_{R}}|Dw|^pdxdt \ \ \text{and} \ \ \fint_{Q^\lambda_{2R}}|Dw|^pdxdt\leq \kappa \lambda^p,
	\end{equation}
	for some $\kappa>1$. Then there exist $\epsilon_0>0$ such that 
	$$
	\Big(\fint_{Q^\lambda_{R}}|Dw|^{p+\epsilon_0}dxdt\Big)^{\f{1}{p+\epsilon_0}}\leq C\fint_{Q^\lambda_{2R}}|Dw|dxdt,
	$$
	where $C$ depends on $n, p, \Lambda_1, \Lambda_2$ and $\kappa$.
\end{prop}
\begin{proof}
	We refer to Corollary 4.8 in \cite{Ba} for the proof of the proposition.
\end{proof}
Let $w$ be a weak solution to \eqref{AppProb1-interior} satisfying \eqref{eq- Dw condition}. We now consider the following problem 
\begin{equation}\label{AppProb2-interior}
\left\{
\begin{aligned}
&v_t-\di\, \overline{\mathbf{a}}_{B_{R}}(D v,t)=0 \quad &\text{in}& \quad Q^\lambda_{R}\equiv Q^\lambda_{R}(z_0),\\
&v=w \quad &\text{on}& \quad \partial_p Q^\lambda_{R},
\end{aligned}\right.
\end{equation}
where $Q_R^\lambda$ is defined by \eqref{defn-QR}.

We then obtain the following estimate.

\begin{lem}
	\label{lem2-inter}
	Let $v$ be a weak solution to \eqref{AppProb2-interior}. Then there exist $C>0$ and $\sigma_1$ so that
	\begin{equation}
	\label{eq-lem2-inter}
	\fint_{Q_{R}^\lambda}|D(w-v)|^pdxdt\leq C[{\bf a}]_{2,R_0}^{\sigma_1}\Big(\fint_{Q_{2R}^\lambda}|Dw|  dxdt\Big)^{p}.
	\end{equation}
\end{lem}
\begin{proof}
	Observe that, by \eqref{eq2-functiona}, we have
	\begin{equation}\label{eq2-proof prop1 inter}
	\begin{aligned}
	\fint_{Q_{R}^\lambda}|D(w-v)|^pdxdt \leq C\fint_{Q_{R}^\lambda} \langle \overline{{\bf a}}_{B_{R}}(Dw,t)-\overline{{\bf a}}_{B_{R}}(Dv,t), Dw-Dv \rangle dxdt.
	\end{aligned}
	\end{equation}
	
	Taking $w-v$ as a test function, it can be verified that
	$$
	\begin{aligned}
	\fint_{Q_{R}^\lambda} \langle \overline{{\bf a}}_{B_{R}}(Dw,t)&-\overline{{\bf a}}_{B_{R}}(Dv,t), Dw-Dv \rangle dxdt\\
	&=\fint_{Q_{R}^\lambda} \langle \overline{{\bf a}}_{B_{R}}(Dw,t)-{{\bf a}}_{B_{R}}(Dw,x,t), Dw-Dv \rangle dxdt.
	\end{aligned}
	$$
	This, in combination with \eqref{eq2-proof prop1 inter}, yields 
	\begin{equation}\label{eq1-proof w-v}
	\begin{aligned}
	\fint_{Q_{R}^\lambda}|D(w-v)|^pdxdt&\leq C\fint_{Q_{R}^\lambda} \langle \overline{{\bf a}}_{B_{R}}(Dw,t)-{{\bf a}}_{B_{R}}(Dw,x,t), Dw-Dv \rangle dxdt\\
	&\leq C\fint_{Q_{R}^\lambda} \Theta(\mathbf{a},B_{R}) |Dw|^{p-1} |D(w-v)|  dxdt.
	\end{aligned}
	\end{equation}
	Applying Young's inequality and Proposition \ref{higherInte-Prop1-inter}, we have, for $\tau>0$,
	\begin{equation}\label{eq2-proof w-v}
	\begin{aligned}
	\fint_{Q_{R}^\lambda} &\Theta(\mathbf{a},B_{R}) |Dw|^{p-1} |D(w-v)|  dxdt\\
	&\leq \tau\fint_{Q_{R}^\lambda}|D(w-v)|^p + C(\tau)\fint_{Q_{R}^\lambda} \Theta(\mathbf{a},B_{R})^{\f{p}{p-1}} |Dw|^p  dxdt\\
	&\leq \tau\fint_{Q_{R}^\lambda}|D(w-v)|^p + C(\tau)\Big(\fint_{Q_{R}^\lambda} \Theta(\mathbf{a},B_{R})^{\f{p(p+\epsilon_0)}{(p-1)\epsilon_0}}dxdt\Big)^{\f{\epsilon_0}{p+\epsilon_0}} \Big(\fint_{Q_{R}^\lambda}|Dw|^{p+\epsilon_0}  dxdt\Big)^{\f{p}{p+\epsilon_0}}\\
	&\leq \tau\fint_{Q_{R}^\lambda}|D(w-v)|^p +C(\tau)[{\bf a}]_{2,R_0}^{\sigma_1}\Big(\fint_{Q_{2R}^\lambda}|Dw|  dxdt\Big)^{p}.
	\end{aligned}
	\end{equation}
	From \eqref{eq1-proof w-v} and \eqref{eq2-proof w-v}, by taking $\tau$ to be sufficiently small, we obtain the desired estimate.
\end{proof}

We now state the standard H\"older regularity result. See for example \cite[Chapter 8]{B}.
\begin{prop}\label{Lvc-Prop-inter}
	Let $v$ solve the equation \eqref{AppProb2-interior}. Then we have
	$$
	\|Dv\|^p_{L^\vc(Q_{R/2}^\lambda)}\leq C\fint_{Q_R^\lambda}|Dv|^pdxdt.
	$$
\end{prop}

We have the following approximation result.

\begin{prop}
	\label{appProp3-inter}
	Let $\mu\in L^{1,\theta}(\Om_T), 1<\theta\leq n+2$. For each $\epsilon>0$ there exists $\delta>0$ so that the following holds true. Assume that $u$ is a weak solution to the problem \eqref{ParabolicProblem} satisfying
	\begin{equation}
	\label{eq-u-prop3}
	\kappa^{-1}\lambda^{p-1}\leq \fint_{Q_{R}^\lambda}|Du|^{p-1}dxdt \ \ \text{and} \ \ \fint_{Q_{4R}^\lambda}|Du|^{p-1}dxdt\leq \kappa\lambda^{p-1},  \ \ \ \text{for some $\kappa>1$,}
	\end{equation}
and 
	\begin{equation}
	\label{eq-mu-prop3}
	\f{|\mu|(Q_{4R}^\lambda)}{|Q_{4R}^\lambda|}\leq \delta\lambda^m.
	\end{equation}
	Then there exists a weak solution $v$ to the problem \eqref{AppProb2-interior} satisfying
	\begin{equation}
	\label{eq-Lvc v-prop3}
	\|Dv\|_{L^\vc(Q_{R/2}^\lambda)}\lesi \lambda,
	\end{equation}
	and
	\begin{equation}
	\label{eq-u-v-prop3-inter}
	\fint_{Q_{R}^\lambda}|D(u-v)|^{p-1}dxdt\leq (\epsilon\lambda)^{p-1}.
	\end{equation}
\end{prop} 
\begin{proof}
	Since $\mu\in L^{1,\theta}$, we have
	$$
	\f{\mu(Q_{4R}^\lambda)}{|Q_{4R}^\lambda|}\leq \f{\mu(Q_{4R})}{|Q_{4R}^\lambda|}\leq \lambda^{p-2}R^{-\theta}.
	$$
	This along with \eqref{eq-mu-prop3} imply
	\begin{equation}\label{eq-mu}
	\begin{aligned}
	\left[\f{|\mu|(Q^\lambda_{4R})}{|Q^\lambda_{4R}|^{(n+1)/(n+2)}}\right]^{\f{n+2}{p+(p-1)n}}&=\left[\f{|\mu|(Q^\lambda_{4R})}{|Q^\lambda_{4R}|}\right]^{\f{n+2}{p+(p-1)n}}|Q^\lambda_{4R}|^{\f{1}{p+(p-1)n}}\\
	&=\left[\f{|\mu|(Q^\lambda_{4R})}{|Q^\lambda_{4R}|}\right]^{\f{1}{\theta}\f{n+2}{p+(p-1)n}}\left[\f{|\mu|(Q^\lambda_{4R})}{|Q^\lambda_{4R}|}\right]^{\f{\theta-1}{\theta}\f{n+2}{p+(p-1)n}}|Q^\lambda_{4R}|^{\f{1}{p+(p-1)n}}\\
	&\lesi \left[\lambda^{p-2}R^{-\theta}\right]^{\f{1}{\theta}\f{n+2}{p+(p-1)n}}\left[\delta \lambda^m\right]^{\f{\theta-1}{\theta}\f{n+2}{p+(p-1)n}}\left[R^{n+2}\lambda^{2-p}\right]^{\f{1}{p+(p-1)n}}\\
	&\lesi \delta^{\f{\theta-1}{\theta}\f{n+2}{p+(p-1)n}}\lambda.
	\end{aligned}
	\end{equation}
	This along with Lemma \ref{lem1-inter} implies that
	\begin{equation}\label{eq-proof Prop 3 interior}
	\fint_{Q_{4R}^\lambda}|D(u-w)|^{p-1}dxdt\leq \mathcal{O}(\delta)\lambda^{p-1}.
	\end{equation}
	Taking this and \eqref{eq-u-prop3} into account, we obtain
	$$
	\lambda^{p-1}\lesi \fint_{Q_{R}^\lambda}|Dw|^{p-1}dxdt, \ \ \ \fint_{Q_{4R}^\lambda}|Dw|^{p-1}dxdt\lesi \lambda^{p-1},
	$$
	provided that  $\delta$ is sufficiently small.
	
	We now apply Proposition 5.5 in \cite{Ba} to find that 
	\begin{equation}\label{eq-wp}
	\bar{\kappa}^{-1} \lambda^{p}\leq \fint_{Q_{R}^\lambda}|Dw|^{p}dxdt, \ \ \ \fint_{Q_{2R}^\lambda}|Dw|^{p}dxdt\leq \bar{\kappa} \lambda^{p},
	\end{equation}
	for some $\bar{\kappa}>1$.
	
	Then the  inequality \eqref{eq-u-v-prop3-inter} follows immediately from  \eqref{eq-proof Prop 3 interior}, Lemma \ref{lem2-inter} and the following estimate
	$$
 	 \fint_{Q^\lambda_{R}}|D(u-v)|^{p-1}dxdt\lesi \fint_{Q^\lambda_{R}}|D(u-w)|^{p-1}dxdt+\fint_{Q^\lambda_{R}}|D(w-v)|^{p-1}dxdt.
	$$
	
	On the other hand, from Proposition \ref{Lvc-Prop-inter} we have
	$$
	\|Dv\|^p_{L^\vc(Q_{R/2}^\lambda)}\lesi\fint_{Q_{R}^\lambda}|Dv|^p dxdt\lesi \fint_{Q_{R}^\lambda}|Dw|^p dxdt+\fint_{Q_{R}^\lambda}|D(w-v)|^p dxdt.
	$$
	This along with \eqref{eq-wp} and Lemma \ref{lem2-inter} yields \eqref{eq-Lvc v-prop3}.
	
\end{proof}
 \section{Boundary estimates}
 
Fix $t_0\in (0, T)$ and $x_0\in \partial\Om$, we set $z_0=(x_0,t_0)$. Let $0<R<R_0/4$ and $\lambda\geq 1$. For the sake of simplicity, we restrict ourself to consider the lateral boundary case with respect to
$$
I^\lambda_{4R}(t_0)\subset (0,T),
$$
since the initial boundary case can be done in the same manner.\\

Before coming to the main comparision estimates, we shall establish some boundary estimates on weak solutions to the homogeneous equations associated to \eqref{ParabolicProblem}.

\subsection{Some boundary estimates for homogeneous equations}
We now consider the  weak solution 
$$
w\in C(I^\lambda_{4R}(t_0); L^2(\Om_{4R}(x_0)))\cap L^p(I^\lambda_{4R}(t_0); W^{1,p}(\Om_{4R}(x_0)))
$$
to the  following equation
\begin{equation}\label{AppProb1-boundary}
\left\{
\begin{aligned}
&w_t-\di\, \mathbf{a}(D w,x,t)=0 \quad &\text{in}& \quad K^\lambda_{4R}(z_0),\\
&w=0 \quad &\text{on}& \quad \partial_wK^\lambda_{4R}(z_0).
\end{aligned}\right.
\end{equation}

 \begin{lem}
 	\label{Caccioppoli-boundary} Let $w$ be a weak solution to the problem \eqref{AppProb1-boundary}. Let $K_{\rho_1}^\lambda(\bar{z})\subset K_{\rho_2}^\lambda(\bar{z})\subset K_{4R}^\lambda(z_0)$ with $\bar{z}=(\bar{x},\bar{t})$ and $\rho_2>\rho_1>0$. Then there exists $c=c(n,p,\Lambda_1, \Lambda_2)$ so that 
 	$$
 	\int_{K^\lambda_{\rho_1}(\bar{z})}|D w|^pdxdt + \sup_{t\in I^\lambda_{\rho_1}(\bar{t})}\int_{B_{\rho_1}(\bar{x})}|w|^2 dx\leq \f{1}{\lambda^{2-p}(\rho_2^2-\rho_1^2)}\int_{K^\lambda_{\rho_2}(\bar{z})}|w|^2dxdt+\f{c}{(\rho_2-\rho_1)^p}\int_{K^\lambda_{\rho_2}(\bar{z})}|w|^pdxdt. 
 	$$
 \end{lem}
 \begin{proof}
 	We adapt an indea in \cite{KL} to our present situation. Fix $t_1\in I^\lambda_{\rho_1}(\bar{t})$. Let $\eta\in C_0^\vc(Q_{\rho_2}^\lambda(\bar{z}))$ such that $\eta\geq 0$, $\eta=1$ in $Q_{\rho_1}^\lambda(\bar{z})$ and
 	\begin{equation}
 	\label{eq-Deta}
 	(\rho_2-\rho_1)|D\eta|+\lambda^{2-p}(\rho^2_2-\rho^2_1)|\eta_t|\leq 100.
 	 	\end{equation}
 	For $\epsilon\in (0,1)$ we define the function $\chi^\epsilon_{t_1}\in C_c^\vc([\epsilon/2,t_1-\epsilon/2])$ with
 	$$
 	\chi^\epsilon_{t_1}(t)=1 \ \ \text{in $[\epsilon,t_1-\epsilon]$, and $\left|[\chi^\epsilon_{t_1}(t)]'\right|\leq 2/\epsilon$}.
 	$$
 	We set $\varphi^\epsilon(x,t)=\eta^p(x,t)w(x,t)\chi^\epsilon_{t_1}(t)$. Taking $\varphi^\epsilon$ as a test function, we obtain
 	\begin{equation}
 	\label{eq-I12}
 	J^\epsilon_1 + J^\epsilon_2:=-\int_{K_{\rho_2}^\lambda(\bar{z})}w\varphi^\epsilon_t dz + \int_{K_{\rho_2}^\lambda(\bar{z})}{\bf a}(Dw,x,t)\cdot D \varphi^\epsilon dz =0.
 	 	\end{equation}
 	By integration by part, we have
 	$$
 	\begin{aligned}
 	J^\epsilon_1&=-\f{1}{2}\int_{K_{\rho_2}^\lambda(\bar{z})}w^2 (\eta^p\chi^\epsilon_{t_1})_t dz\\
 	&= -\f{p}{2}\int_{K_{\rho_2}^\lambda(\bar{z})}w^2 \eta^{p-1}\eta_t\chi^\epsilon_{t_1} dz-\f{1}{2}\int_{K_{\rho_2}^\lambda(\bar{z})}w^2 \eta^{p}(\chi^\epsilon_{t_1})_t dz,
 	\end{aligned}
 	$$
 	which implies 
 	$$
 	J_1^\epsilon \to -\f{p}{2}\int_{B_{\rho_2}(\bar{x})\times (0,t_1)}w^2 \eta^{p-1}\eta_t dz+\f{1}{2}\int_{B_{\rho_2}(\bar{x})}w(x,t_1)^2 \eta^{p}(x,t_1)dx \ \ \ \text{as $\epsilon\to 0$}.
 	$$
 	On the other hand, we have
 	$$
 	J_2^\epsilon \to  \int_{B_{\rho_2}(\bar{x})\times (0,t_1)}[{\bf a}(Dw,x,t)\cdot Dw]\eta^pdz+p\int_{B_{\rho_2}(\bar{x})\times (0,t_1)}[{\bf a}(Dw,x,t)\cdot D\eta]\eta^{p-1}wdz \ \ \ \text{as $\epsilon\to 0$}.
 	$$
 	Taking \eqref{eq-I12} and these two estimates above into account we find that
 	$$
 	\begin{aligned}
 	\int_{B_{\rho_2}(\bar{x})\times (0,t_1)}&[{\bf a}(Dw,x,t)\cdot Dw]\eta^pdz+\f{1}{2}\int_{B_{\rho_2}}w(x,t_1)^2 \eta^{p}(x,t_1)dx\\
 	&\leq \f{p}{2}\int_{B_{\rho_2}(\bar{x})\times (0,t_1)}w^2 \eta^{p-1}|\eta_t| dz+p\int_{B_{\rho_2}(\bar{x})\times (0,t_1)}\left|[{\bf a}(Dw,x,t)\cdot D\eta]\right|\eta^{p-1}wdz.
 	 	\end{aligned}
 	$$
 	This together with \eqref{eq1-functiona}, \eqref{eq2-functiona} and \eqref{eq-Deta} implies that
 	$$
 	\begin{aligned}
 	\int_{B_{\rho_2}(\bar{x})\times (0,t_1)}|Dw|^p\eta^pdz&+\f{1}{2}\int_{B_{\rho_2}(\bar{x})}w(x,t_1)^2 \eta^{p}(x,t_1)dx\\
 	&\lesi \f{1}{\lambda^{2-p}(\rho_2^2-\rho_1^2)}\int_{K_{\rho_2}^\lambda(\bar{z})}w^2 dz+p\int_{B_{\rho_2}(\bar{x})\times (0,t_1)}|Dw|^{p-1}|D\eta|\eta^{p-1}|w|dz.
 	\end{aligned}
 	$$
 	Applying Young's inequality we deduce that, for $\tau>0$,
 	$$
 	\begin{aligned}
 	\int_{B_{\rho_2}(\bar{x})\times (0,t_1)}&|Dw|^p\eta^pdz+\f{1}{2}\int_{B_{\rho_2}(\bar{x})}w(x,t_1)^2 \eta^{p}(x,t_1)dx\\
 	&\lesi \f{1}{\lambda^{2-p}(\rho_2^2-\rho_1^2)}\int_{K_{\rho_2}^\lambda(\bar{z})}w^2 dz+\tau \int_{B_{\rho_2}(\bar{x})\times (0,t_1)}|Dw|^p\eta^pdz +c(\tau)\int_{B_{\rho_2}(\bar{x})\times (0,t_1)}|D\eta|^p |w|^pdz\\
 	&\lesi \f{1}{\lambda^{2-p}(\rho_2^2-\rho_1^2)}\int_{K_{\rho_2}^\lambda(\bar{z})}w^2 dz+\tau \int_{B_{\rho_2}(\bar{x})\times (0,t_1)}|Dw|^p\eta^pdz +\f{c(\tau)}{(\rho_2-\rho_1)^p}\int_{B_{\rho_2}(\bar{x})\times (0,t_1)}|w|^pdz\\
 	\end{aligned}
 	$$
 	By choosing $\tau$ to be sufficiently small, we end up with
 	$$
 	\begin{aligned}
 	\int_{B_{\rho_2}(\bar{x})\times (0,t_1)}&|Dw|^p\eta^pdz+\f{1}{2}\int_{B_{\rho_2}(\bar{x})}w(x,t_1)^2 \eta^{p}(x,t_1)dx\\
 	&\lesi \f{1}{\lambda^{2-p}(\rho_2^2-\rho_1^2)}\int_{K_{\rho_2}^\lambda(\bar{z})}w^2 dz+ \f{1}{(\rho_2-\rho_1)^p}\int_{K_{\rho_2}^\lambda(\bar{z})}|w|^pdz.
 	\end{aligned}
 	$$
 	This deduces the desired estimate.
 \end{proof}
 
 We now give a useful result which will be used in the sequel.
 \begin{lem}
 	\label{lem-subslolution} Let $w$ be a weak solution to the equation
\eqref{AppProb1-boundary}. 	Then for $\theta\in (0,1)$ and $K_{\rho,\sigma}(z_0)\subset K^\lambda_{4R}(z_0)$ with $\rho, \sigma>0$   we have
 	\begin{equation}\label{eq-subsolution}
 	\sup_{K_{\theta \rho,\theta \sigma}(z_0)} |w|\leq c\left(\f{1}{(1-\theta)}\right)^{n+p}\f{
 		\sigma}{\rho^p}\fint_{K_{\rho,\sigma}(z_0)}|w|^{p-1}dz + \Big(\f{\rho^p}{\sigma}\Big)^{\f{1}{p-2}}.
 	\end{equation}
 \end{lem}
\begin{proof}
	Since $w$ is a weak solution to \eqref{AppProb1-boundary}, $|w|$ is a nonnegative subsolution to the equation \eqref{AppProb1-boundary}. See for example Lemma 1.1 in \cite[p. 19]{B}. Recall that a sub-solution is a function such that the left-hand side of the weak formula of \eqref{AppProb1-boundary} is negative, for all positive test functions.
	
	The estimate \eqref{eq-subsolution} was proved in Theorem 4.1 in \cite[pp.122--123]{B} for the interior case. The estimate is still true near the boundary of a Reifenberg domain by similar argument  with some minor modifications. Hence, we skip the proof of \eqref{eq-subsolution} here and leave it to interested readers. 
\end{proof}
 
 \begin{prop}
 	\label{prop1- imply u p}
 	Let $w$ be a weak solution to the problem \eqref{AppProb1-boundary} satisfying the estimates
 	 	\begin{equation}\label{eq-lambda bound for w}
 	\kappa^{-1}\lambda^{p}\leq \fint_{K^\lambda_{R}(z_0)}|Dw|^{p}dxdt \ \ \text{and} \ \ \fint_{K^\lambda_{2R}(z_0)}|Dw|^{p}dxdt \leq \kappa\lambda^{p},
 	\end{equation}
 	for some $\kappa\geq 1$.
 	
 	Then there exist $1\leq q<p$, $c=c(n,p,\Lambda_1,\Lambda_2, \kappa)$ and $\sigma=\sigma(n,p)$ so that 
 	$$
 	\Big(\fint_{K_{r_1}^\lambda(z_0)}|Dw|^pdxdt\Big)^{1/p}\leq c\Big(\f{2R}{r_2-r_1}\Big)^\sigma \Big(\fint_{K_{r_2}^\lambda(z_0)}|Dw|^qdxdt\Big)^{1/q},
 	$$
 	for all $R\leq  r_1<r_2\leq 2R$.
 \end{prop}
 \begin{proof}
 	For the sake of simplicity, we shall write, respectively, $K_r^\lambda, \Om_r$ for $K_r^\lambda(z_0), \Om_r(x_0)$ for all $r>0$. Set $r_3=r_1 +(r_2-r_1)/2$. Then from Lemma \ref{Caccioppoli-boundary}, we have
 	$$
 	\begin{aligned}
 	\fint_{K_{r_1}^\lambda}|Dw|^pdxdt&\lesi \f{1}{\lambda^{2-p}(r_3^2-r_1^2)}\fint_{K^\lambda_{r_3}}|w|^2dxdt+\f{1}{(r_3-r_1)^p}\fint_{K^\lambda_{r_3}}|w|^pdxdt\\
 	&\lesi \f{1}{\lambda^{2-p}(r_3-r_1)^2}\fint_{K^\lambda_{r_3}}|w|^2dxdt+\f{1}{(r_3-r_1)^p}\fint_{K^\lambda_{r_3}}|w|^pdxdt\\
 	&\sim \f{1}{\lambda^{2-p}(r_2-r_1)^2}\fint_{K^\lambda_{r_3}}|w|^2dxdt+\f{1}{(r_2-r_1)^p}\fint_{K^\lambda_{r_3}}|w|^pdxdt\\
 	&:=I_1+I_2.
 	\end{aligned}
 	$$
 	By H\"older's inequality, for $\tau>0$ we have
 	$$
 	\begin{aligned}
 	I_1&\leq \lambda^{p-2}\Big(\f{1}{(r_2-r_1)^p}\fint_{K^\lambda_{r_3}}|w|^pdxdt\Big)^{2/p}\\
 	&\leq \tau \lambda^p + c(\tau)I_2
 	\end{aligned}
 	$$
 	where in the last inquality we used Young's inequality.
 	
 	From this and \eqref{eq-lambda bound for w}, by taking $\tau$ to be sufficiently small, we find that
 	$$
 	\fint_{K_{r_1}^\lambda}|Dw|^pdxdt\lesi I_2.
 	$$
 	Hence, it suffices to prove that  
 	\begin{equation}\label{eq-I2}
 	I_2\lesi \Big(\f{2R}{r_2-r_1}\Big)^{p\sigma} \Big(\fint_{K_{r_2}^\lambda}|Dw|^qdxdt\Big)^{p/q}.
 	\end{equation}
 	Indeed, we now consider two cases: $2\leq p<n+2$ and $p\geq n+2$.\\

 	\noindent{\bf Case 1: $2\leq p<n+2$.} By H\"older's inequality, we have
 	$$
 	\fint_{\Om_{r_3}}|w|^pdx\leq \Big(\fint_{\Om_{r_3}}|w|^2dx\Big)^{q/n}\Big(\fint_{\Om_{r_3}}|w|^{q^*}dx\Big)^{q/q^*},
 	$$
 	where $q=pn/(n+2)<\min\{n,p\}$ and $q^*=nq/(n-q)$. 
 	
Then applying Sobolev-Poincar\'e's inequalities \eqref{SP-inequality 1}, we have
 	$$
 	 \Big(\fint_{\Om_{r_3}}|w|^{q^*}dx\Big)^{q/q^*}\lesi r_3^{q}\fint_{\Om_{r_3}}|Dw|^{q}dx.
 	$$
 	
 	Hence,
 	$$
 	\begin{aligned}
 		\fint_{\Om_{r_3}}|w|^pdx&\leq r_3^{q}\Big(\fint_{\Om_{r_3}}|w|^2dx\Big)^{q/n}\Big(\fint_{\Om_{r_3}}|Dw|^{q}dx\Big)\\
 		&\sim \Big(\int_{\Om_{r_3}}|w|^2dx\Big)^{q/n}\Big(\fint_{\Om_{r_3}}|Dw|^{q}dx\Big).
 	\end{aligned}
 	$$
 	This implies that 
 	\begin{equation}\label{eq1-I2}
 	\begin{aligned}
 	I_2&\lesi \f{1}{(r_2-r_1)^p}\Big(\fint_{K_{r_3}^\lambda}|Dw|^{q}dz\Big)\Big(\sup_{t\in I^\lambda_{r_3}}\int_{\Om_{r_3}}|w|^2dx\Big)^{q/n}\\
 	&\lesi r_2^{-p}\Big(\f{R}{r_2-r_1}\Big)^p\Big(\fint_{K_{r_3}^\lambda}|Dw|^{q}dz\Big)\Big(\sup_{t\in I^\lambda_{r_3}}\int_{\Om_{r_3}}|w|^2dx\Big)^{q/n}.
 	\end{aligned}
 	\end{equation}
 	
 	On the other hand, by Lemma \ref{Caccioppoli-boundary} and H\"older's inequality, we have
 	$$
 	\begin{aligned}
 	\sup_{t\in I^\lambda_{r_3}}\int_{\Om_{r_3}}|w|^2dx &\lesi \f{1}{\lambda^{2-p}(r_2^2-r_3^2)}\int_{K^\lambda_{r_2}}|w|^2dxdt+\f{c}{(r_2-r_3)^p}\int_{K^\lambda_{r_2}}|w|^pdxdt\\
 	&\lesi \f{1}{\lambda^{2-p}(r_2-r_3)^2}\int_{K^\lambda_{r_2}}|w|^2dxdt+\f{1}{(r_2-r_3)^p}\int_{K^\lambda_{r_3}}|w|^pdxdt\\
 	&\sim \f{1}{\lambda^{2-p}(r_2-r_1)^2}\int_{K^\lambda_{r_2}}|w|^2dxdt+\f{1}{(r_2-r_1)^p}\int_{K^\lambda_{r_2}}|w|^pdxdt\\
 	&\lesi r_2^{n+2}\Big(\f{1}{(r_2-r_1)^p}\fint_{K^\lambda_{r_2}}|w|^pdxdt\Big)^{2/p}+\f{\lambda^{2-p}r_2^{n+2}}{(r_2-r_1)^p}\fint_{K^\lambda_{r_2}}|w|^pdxdt.
 	\end{aligned}
 	$$
 	Applying Sobolev-Poincar\'e's inequality \eqref{SP-inequality 2} and \eqref{eq-lambda bound for w}, we obtain further
 	$$
 	\begin{aligned}
 	\sup_{t\in I^\lambda_{r_3}}\int_{\Om_{r_3}}|w|^2dx &\lesi r_2^{n+2}\Big(\f{r_2^p}{(r_2-r_1)^p}\fint_{K^\lambda_{r_3}}|Dw|^pdxdt\Big)^{2/p}+\lambda^{2-p}r_2^{n+2}\f{r_2^p}{(r_2-r_1)^p}\fint_{K^\lambda_{r_2}}|Dw|^pdxdt\\
 	&\lesi r_2^{n+2}\lambda^2\Big(\f{2R}{r_2-r_1}\Big)^p.
 	\end{aligned}
 	$$
 	Inserting this into \eqref{eq1-I2}, and then using Young's inequality we obtain, for $\tau>0$,
 	$$
 	\begin{aligned}
 	I_2&\lesi r_2^{(n+2)q/n-p}\lambda^{2q/n}\Big(\f{2R}{r_2-r_1}\Big)^{p+qp/n}\fint_{K_{r_3}^\lambda}|Dw|^{q}dz\\
 	&=\Big(\f{2R}{r_2-r_1}\Big)^{p+qp/n}\lambda^{\f{2p}{n+2}}\fint_{K_{r_3}^\lambda}|Dw|^{q}dz\\
 	&\leq \tau \lambda^p + c(\tau)\Big(\f{2R}{r_2-r_1}\Big)^{\f{p^2(n+q)}{pq}}\Big(\fint_{K_{r_3}^\lambda}|Dw|^{q}dz\Big)^{p/q}.
 	\end{aligned}
 	$$
 	This together with the fact that $I_2\geq C\lambda^p$ implies that
 	$$
 	I_2 \leq c(\tau)\Big(\f{2R}{r_2-r_1}\Big)^{\f{n+2}{n}+\f{p^2}{n}}\Big(\fint_{K_{r_3}^\lambda}|Dw|^{q}dz\Big)^{p/q}
 	$$
 	provided that $\tau$ is sufficiently small.
 	
 	\medskip
 	
 	\noindent{\bf Case 2: $p\geq n+2$.} By H\"older's inequality, we have
 	$$
 	\fint_{\Om_{r_3}}|w|^pdx\leq \Big(\fint_{\Om_{r_3}}|w|^2dx\Big)^{1/2}\Big(\fint_{\Om_{r_3}}|w|^{2q}dx\Big)^{1/2}
 	$$
where $q=p-1>n$. 

Then applying Sobolev-Poincar\'e's inequalities \eqref{SP-inequality 1}, we have
 	$$
 	\Big(\fint_{\Om_{r_3}}|w|^{2q}dx\Big)^{1/2}= \Big(\fint_{\Om_{r_3}}|w|^{2q}dx\Big)^{\f{q}{2q}}\lesi r_3^{q}\fint_{\Om_{r_3}}|Dw|^{q}dx.
 	$$

 	Hence,
 	$$
 	\begin{aligned}
 	\fint_{\Om_{r_3}}|w|^pdx&\leq r_3^{q}\Big(\fint_{\Om_{r_3}}|w|^2dx\Big)^{1/2}\Big(\fint_{\Om_{r_3}}|Dw|^{q}dx\Big)\\
 	&\sim r_3^{q-n/2}\Big(\int_{\Om_{r_3}}|w|^2dx\Big)^{1/2}\Big(\fint_{\Om_{r_3}}|Dw|^{q}dx\Big).
 	\end{aligned}
 	$$
 	Therefore, 
 	\begin{equation}\label{eq1-I2-case2}
 	\begin{aligned}
 	I_2&\lesi \f{r_3^{q-n/2}}{(r_2-r_1)^p}\Big(\fint_{K_{r_3}^\lambda}|Dw|^{q}dz\Big)\Big(\sup_{t\in I^\lambda_{r_3}}\int_{\Om_{r_3}}|w|^2dx\Big)^{1/2}\\
 	&\lesi r_2^{-1-n/2}\Big(\f{R}{r_2-r_1}\Big)^p\Big(\fint_{K_{r_3}^\lambda}|Dw|^{q}dz\Big)\Big(\sup_{t\in I^\lambda_{r_3}}\int_{\Om_{r_3}}|w|^2dx\Big)^{1/2}.
 	\end{aligned}
 	\end{equation}
 	
 	In Case 1, we proved that
 	$$
 	\begin{aligned}
 	\sup_{t\in I^\lambda_{r_3}}\int_{\Om_{r_3}}|w|^2dx 
 	&\lesi r_2^{n+2}\lambda^2\Big(\f{2R}{r_2-r_1}\Big)^p.
 	\end{aligned}
 	$$
 	Inserting this into \eqref{eq1-I2-case2}, and then using Young's inequality we obtain, for $\tau>0$,
 	$$
 	\begin{aligned}
 	I_2&\lesi r_2^{(n+2)/2 -1-n/2}\lambda\Big(\f{2R}{r_2-r_1}\Big)^{3p/2}\fint_{K_{r_3}^\lambda}|Dw|^{q}dz\\
 	&=\lambda\Big(\f{2R}{r_2-r_1}\Big)^{3p/2}\fint_{K_{r_3}^\lambda}|Dw|^{q}dz\\
 	&\leq \tau \lambda^p + c(\tau)\Big(\f{2R}{r_2-r_1}\Big)^{\f{3p^2}{2q}}\Big(\fint_{K_{r_3}^\lambda}|Dw|^{q}dz\Big)^{\f{p}{q}}.
 	\end{aligned}
 	$$
 	This together with the fact that $I_2\geq C\lambda^p$ implies that
 	$$
 	I_2 \lesi c(\tau)\Big(\f{2R}{r_2-r_1}\Big)^{\f{3p^2}{2(p-1)}}\Big(\fint_{K_{r_3}^\lambda}|Dw|^{q}dz\Big)^{\f{p}{q}}
 	$$
 	provided that $\tau$ is sufficiently small.
 	
 	This completes our proof.
 \end{proof}
 
 We now recall the following result in \cite[Lemma 5.1]{KM1}.
 
 \begin{lem}
 	\label{lem-KM}
 	Let $1<q<p<\vc$ and $\sigma\geq 0$, and let $\{U_{\theta}: 0<\theta\leq 1\}$ be a family of open sets in $\mathbb{R}^{n+1}$ with property $U_{\theta_1}\subset U_{\theta_2}\subset U_1\equiv U$ whenever $0<\theta_1\leq \theta_2<1$. If $f\in L^q(U)$ is a non-negative function satisfying 
 	$$
 	\Big(\fint_{U_{\theta_1}}f^p dxdt\Big)^{1/p}\leq \f{c_0}{(\theta_2-\theta_1)^\sigma}\Big(\fint_{U_{\theta_2}}f^q dxdt\Big)^{1/q},
 	$$
 	for all $1/2\leq \theta_1<\theta_2\leq 1$, then there exists $c=c(c_0,\sigma,p,q)$ so that
 	$$
 	\Big(\fint_{U_{\theta}}f^p dxdt\Big)^{1/p}\leq\f{c}{(1-\theta)^{\f{\sigma q(p-1)}{p-q}}}\fint_{U}f dxdt.
 	$$
 \end{lem}
As a direct consequence of Proposition \ref{prop1- imply u p} and Lemma \ref{lem-KM}, we deduce the following result.
 \begin{lem}
 	\label{lem- imply u p}
 	Let $w$ be a weak solution to the problem \eqref{AppProb1-boundary} satisfying the estimates
 	\begin{equation*}
 	\kappa^{-1}\lambda^{p}\leq \fint_{K^\lambda_{R}(z_0)}|Dw|^{p}dxdt \ \ \text{and} \ \ \fint_{K^\lambda_{2R}(z_0)}|Dw|^{p}dxdt \leq \kappa\lambda^{p},
 	\end{equation*}
 	for some $\kappa\geq 1$.
 	
 	Then we have 
 	$$
 	\Big(\fint_{K_{R}^\lambda(z_0)}|Dw|^pdxdt\Big)^{1/p}\lesi \fint_{K_{2R}^\lambda(z_0)}|Dw|dxdt.
 	$$
 \end{lem}
 
\begin{prop}
	\label{prop2- imply u p}
	Let $w$ be a weak solution to the problem \eqref{AppProb1-boundary} satisfying the estimates
	\begin{equation}\label{eq1-lambda bound for w}
	\f{1}{\kappa_1}\lambda^{p-1}\leq \fint_{K^\lambda_R(z_0)}|Dw|^{p-1}dxdt \ \ \text{and} \ \ \fint_{K^\lambda_{4R}(z_0)}|Dw|^{p-1}dxdt\leq \kappa_1\lambda^{p-1},
	\end{equation}
	for some $\kappa_2\geq 1$ and $\lambda>1$. Then we have
	\begin{equation}\label{eq2-lambda bound for w}
	\f{1}{\kappa_2}\lambda^{p}\leq \fint_{K^\lambda_{R}(z_0)}|Dw|^{p}dxdt \ \ \text{and} \ \ \fint_{K^\lambda_{2R}(z_0)}|Dw|^{p}dxdt\leq \kappa_2\lambda^{p}.
	\end{equation}
	\end{prop}
\begin{proof}
	By H\"older's inequality, we have
	$$
	\fint_{K^\lambda_{R}(z_0)}|Dw|^{p}dxdt\geq C\lambda^{p}.
	$$
	It remains to prove the second inequality in \eqref{eq2-lambda bound for w}. Indeed, from Lemma \ref{Caccioppoli-boundary} we have
	$$
	\fint_{K^\lambda_{2R}(z_0)}|D w|^pdxdt \leq \f{c}{\lambda^{2-p} R^2}\fint_{K^\lambda_{3R}(z_0)}|w|^2dxdt+\f{c}{R^p}\fint_{K^\lambda_{3R}(z_0)}|w|^pdxdt. 
	$$
	Applying H\"older's inequality and Young's inequality, we deduce
	$$
	\begin{aligned}
	\fint_{K^\lambda_{2R}(z_0)}|D w|^pdxdt &\leq \f{c}{\lambda^{2-p} R^2}\Big(\fint_{K^\lambda_{3R}(z_0)}|w|^pdxdt\Big)^{2/p}+\f{c}{R^p}\fint_{K^\lambda_{3R}(z_0)}|w|^pdxdt\\
	&\lesi \lambda^p+\f{1}{R^p}\fint_{K^\lambda_{3R}(z_0)}|w|^pdxdt\\
	&\lesi \lambda^p +\f{1}{R^p}\sup_{K^\lambda_{3R}(z_0)}|w|^p.
	\end{aligned}
	$$
	Hence, by using Lemma \ref{lem-subslolution} with $\theta=3/4, \rho=4R$ and $\sigma = \lambda^{2-p}(4R)^2$, we obtain
	$$
	\sup_{K^\lambda_{3R}(z_0)}|w|\lesi \f{\lambda^{2-p}}{R^{p-2}}\fint_{Q^\lambda_{4R}(z_0)}|w|^{p-1}dxdt+ R\lambda.
	$$
	By Sobolev-Poincar\'e's inequality \eqref{SP-inequality 2}, we further obtain
	$$
	\sup_{K^\lambda_{3R}(z_0)}|w|\lesi R\lambda^{2-p}\fint_{Q^\lambda_{4R}(z_0)}|Dw|^{p-1}dxdt+ R\lambda\lesi R\lambda.
	$$
	Hence,
	$$
	\fint_{K^\lambda_{2R}(z_0)}|D w|^pdxdt \lesi \lambda^p.
	$$
	This completes our proof.
\end{proof}

\begin{prop}
	\label{higherInte-Prop1-inter}
	Let $w$ be a weak solution to the problem \eqref{AppProb1-boundary}. Assume that 
	\begin{equation}\label{eq-Dup higher inter}
	\kappa^{-1} \lambda^p\leq \fint_{K^\lambda_{R}(z_0)}|Dw|^pdxdt \ \ \text{and} \ \ \fint_{K^\lambda_{2R}(z_0)}|Dw|^pdxdt\leq \kappa \lambda^p,
	\end{equation}
	for some $\kappa>1$. Then there exists $\epsilon_0>0$ so that 
	$$
	\Big(\fint_{K^\lambda_{R}(z_0)}|Dw|^{p+\epsilon_0}dxdt\Big)^{\f{1}{p+\epsilon_0}}\leq C\fint_{K^\lambda_{2R}(z_0)}|Dw|dxdt.
	$$
\end{prop}
\begin{proof}
We now consider
	the rescaled maps
	\begin{equation}\label{Scaledmap1}
	\left\{
	\begin{aligned}
	&\bar{w}(x,t)= \f{u(x_0+R_ix,t_0+\lambda^{2-p}R^2t)}{R\lambda},\\
	&\bar{{\bf a}}_i(\xi,x,t)= \f{{\bf a}(\lambda \xi,x_0+Rx,t_0+\lambda^{2-p}R^2t)}{\lambda^{p-1}}.
	\end{aligned}\right.
	\end{equation}
	Then arguing similarly to the proof of Theorem 4.7 in \cite{P2}, we obtain
		\begin{equation}\label{eq1-prop2.2}
	\Big(\fint_{K_{1}}|D\bar{w}|^{p+\epsilon_0}dxdt\Big)^{\f{1}{p+\epsilon_0}}\leq C\Big(\fint_{K_{2}}|D\bar{w}|^{p}dxdt\Big)^{\sigma}
	\end{equation}
		where $\sigma = (2+\epsilon_0)/(2(p+\epsilon_0))$.
		
	Rescaling back in \eqref{eq1-prop2.2} we get that
	$$
	\Big(\fint_{K_{R}^\lambda}|Dw|^{p+\epsilon_0}dxdt\Big)^{\f{1}{p+\epsilon_0}}\leq C\lambda^{1-\sigma p}\Big(\fint_{K_{2R}^\lambda}|Dw|^{p}dxdt\Big)^{\sigma}.
	$$
 	This together with \eqref{eq-Dup higher inter} implies the desired estimate.
\end{proof}

We now give some comparision estimates for the weak solutions to \eqref{ParabolicProblem}.

\subsection{Comparision estimates}

Assume that $u$ is a weak solution to the problem \eqref{ParabolicProblem}. We consider the following equation
\begin{equation}\label{AppProb1b-boundary}
\left\{
\begin{aligned}
&w_t-\di\, \mathbf{a}(D w,x,t)=0 \quad &\text{in}& \quad K^\lambda_{4R}(z_0),\\
&w=u \quad &\text{on}& \quad \partial_p K^\lambda_{4R}(z_0).
\end{aligned}\right.
\end{equation}
It is well-known that $w$ exists and unique.

Arguing similarly to the proof of Lemma 4.1 in \cite{KM}, we can prove the following estimate.
\begin{lem}\label{lem1-boundary}
	Let $w$ be a weak solution to the problem  \eqref{AppProb1b-boundary}. Then for every $1\leq q<p-1+\f{1}{n+1}$, there exists $C$ so that
	\begin{equation}
		\label{eq1-lem1-boundary}
		\Big(\fint_{K^\lambda_{4R}(z_0)}|D(u-w)|^qdxdt\Big)^{1/q}\leq C\left[\f{|\mu|(K^\lambda_{4R}(z_0))}{|K^\lambda_{4R}(z_0)|^{(n+1)/(n+2)}}\right]^{\f{n+2}{p+(p-1)n}}.
	\end{equation}
\end{lem}

We now assume that $0<\delta<1/50$. Since $x_0\in \partial\Om$, there exists a new coordinate system whose variables are still denoted by $(x_1,\ldots,x_n)$ such that in this coordinate system the origin is some interior point of $\Om$, $x_0=(0,\ldots,0,-\f{\delta R}{2(1-\delta)})$ and
\begin{equation}\label{eq1-new coordinate}
B_{R/2}^+ \subset B_{R/2}\cap \Om\subset B_{R/2}\cap \{x: x_n>-3\delta R\}.
\end{equation}
Note that due to $\delta\in (0,1/50)$, we further obtain 
\begin{equation}\label{eq2-new coordinate}
B_{3R/8}\subset B_{R/4}(x_0)\subset B_{R/2}\subset B_R(x_0).
\end{equation}

Let $w$ be a weak solution to \eqref{AppProb1b-boundary} satisfying 
\begin{equation}
\label{eq- Dw condition boundary}
\f{1}{\kappa_2}\lambda^{p}\leq \fint_{K^\lambda_{R}(z_0)}|Dw|^{p}dxdt \ \ \text{and} \ \ \fint_{K^\lambda_{2R}(z_0)}|Dw|^{p}dxdt\leq \kappa_2\lambda^{p}.
\end{equation}

We now consider the following problem (in the new coordinate system)
\begin{equation}\label{AppProb2-boundary}
\left\{
\begin{aligned}
&h_t-\di\, \overline{\mathbf{a}}_{B_{R/2}}(D h,t)=0 \quad &\text{in}& \quad K_{R/2}^\lambda(0,t_0),\\
&h=w \quad &\text{on}& \quad \partial_p K_{R/2}^\lambda(0,t_0).
\end{aligned}\right.
\end{equation}
Using the argument as in the proof of Lemma \ref{lem2-inter} and the fact that $B_R\subset B_{2R}(x_0)$ we obtain the following estimate.
\begin{lem}
	\label{lem2-boundary}
	Let $h$ be a weak solution to \eqref{AppProb2-boundary}. Then there exist $C>0$ and $\sigma_2$ so that
	\begin{equation}
	\label{eq-lem2-boundary}
	\fint_{K_{R/2}^\lambda(0,t_0)}|D(w-h)|^pdxdt\leq C[{\bf a}]_{2,R_0}^{\sigma_2}\Big(\fint_{K_{2R}^\lambda(z_0)}|Dw|  dxdt\Big)^{p}.
	\end{equation}
\end{lem}
The main different from the interior case is that due to the lack of smoothness condition on the boundary of $\Om$, we can not expect that the $L^\vc$-norm of $Dh$ is finite near the boundary. To handle this trouble, we consider its associated problem.
\begin{equation}\label{AppProb3-boundary}
\left\{
\begin{aligned}
&v_t-\di\, \overline{\mathbf{a}}_{B_{R/2}}(D v,t)=0 \quad &\text{in}& \quad (Q_{R/2}^\lambda)^+(0,t_0),\\
&v=0 \quad &\text{on}& \quad Q_{R/2}^\lambda(0,t_0)\cap \{z=(x',x_n,t): x_n=0\}.
\end{aligned}\right.
\end{equation}

 \begin{prop}
 	\label{approxPropBoundary-1}
 	Let $\mu\in L^{1,\theta}(\Om_T), 1<\theta\leq n+2$. For each $\epsilon>0$ there exists $\delta>0$ so that the following holds true. Assume that $u$ is a weak solution to the problem \eqref{ParabolicProblem} satisfying
 	\begin{equation}
 	\label{eq-u-prop3-boundary}
 	\kappa^{-1}\lambda^{p-1}\leq \fint_{K_{R}^\lambda(z_0)}|Du|^{p-1}dxdt, \ \ \ \fint_{K_{4R}^\lambda(z_0)}|Du|^{p-1}dxdt\leq \kappa\lambda^{p-1},  \ \ \ \text{for some $\kappa>1$,}
 	\end{equation}
 	and 
 	\begin{equation}
 	\label{eq-mu-prop-boundary3}
 	\f{|\mu|(K_{4R}^\lambda(z_0))}{|K_{4R}^\lambda(z_0)|}\leq \delta\lambda^m.
 	\end{equation}
 	Then there exists a weak solution $v$ to the problem \eqref{AppProb3-boundary} satisfying
 	\begin{equation}
 	\label{eq-Lvc v-prop3-boundary}
 	\|D\bar{v}\|_{L^\vc(Q_{R/8}^\lambda(z_0))}\lesi \lambda,
 	\end{equation}
 	and
 	\begin{equation}
 	\label{eq-u-v-prop3}
 	\fint_{K_{R/4}^\lambda(z_0)}|D(u-\bar{v})|^{p-1}dxdt\leq (\epsilon\lambda)^{p-1}
 	\end{equation}
 	where $\bar{v}$ is the zero extension of $v$ to $Q_{R/2}^\lambda(0,t_0)\supset Q_{R/4}^\lambda(z_0)$.
 \end{prop}
 \begin{proof}
 	Similarly to \eqref{eq-mu}, we have
 	$$
 	\left[\f{|\mu|(K^\lambda_{4R}(z_0))}{|K^\lambda_{4R}(z_0)|^{(n+1)/(n+2)}}\right]^{\f{n+2}{p+(p-1)n}}\lesi \mathcal{O}(\delta)\lambda.
    $$ 
This along with Lemma \ref{lem1-boundary} implies that
\begin{equation}\label{eq-proof Prop 3 boundary}
\fint_{K_{4R}^\lambda(z_0)}|D(u-w)|^{p-1}dxdt\leq \mathcal{O}(\delta)\lambda^{p-1}.
\end{equation}
From this inequality and \eqref{eq-u-prop3-boundary}, we find that 
$$
\lambda^{p-1}\lesi \fint_{K_{R}^\lambda(z_0)}|Dw|^{p-1}dxdt, \ \ \ \fint_{K_{4R}^\lambda(z_0)}|Dw|^{p-1}dxdt\lesi \lambda^{p-1}
$$
provided that $\delta$ is sufficiently small.

Applying Proposition \ref{prop2- imply u p}, we obtain
\begin{equation}\label{eq- wp}
\kappa^{-1}\lambda^{p}\leq \fint_{K^\lambda_{R}(z_0)}|Dw|^{p}dxdt, \ \ \ \fint_{K^\lambda_{2R}(z_0)}|Dw|^{p}dxdt\leq \kappa\lambda^{p}
\end{equation}
for some $\kappa>1$.

This together with Lemma  \ref{lem2-boundary} implies that if $h$ is a solution to \eqref{AppProb2-boundary}, then it also solves
 \begin{equation}\label{AppProb2s-boundary}
 \left\{
 \begin{aligned}
 &h_t-\di\, \overline{\mathbf{a}}_{B_{R/2}}(D h,t)=0 \quad &\text{in}& \quad K_{R/2}^\lambda(0,t_0),\\
 &h=0 \quad &\text{on}& \quad \partial_w K_{R/2}^\lambda(0,t_0),
 \end{aligned}\right.
 \end{equation}
 with
 $$
 \fint_{K_{R/2}^\lambda(0,t_0)}|h|^pdz \lesi \fint_{K_{R/2}^\lambda(0,t_0)}|h-w|^pdz+\fint_{K_{R/2}^\lambda(0,t_0)}|w|^pdz\lesi \lambda^p.
 $$
 We first show that there exists a weak solution $v$ to the problem \eqref{AppProb3-boundary} such that 
\begin{equation}\label{eq1a-prop3}
 \|D\bar{v}\|_{L^\vc(Q_{R/4}^\lambda(0,t_0))}\lesi \lambda,
\end{equation}
and
\begin{equation}\label{eq1b-prop3}
\fint_{K_{3R/8}^\lambda(0,t_0)}|D(h-\bar{v})|^{p}dxdt\leq (\epsilon\lambda)^{p}
\end{equation}
where $\bar{v}$ is the zero extension of $v$ to $Q_{R/2}^\lambda(0,t_0)$.

Once \eqref{eq1a-prop3} and \eqref{eq1b-prop3} are proved, the desired estimates follow immediately. Indeed, assume that \eqref{eq1a-prop3} and \eqref{eq1b-prop3} hold true. Since $K_{R/4}^\lambda(z_0)\subset K_{R/2}^\lambda(0,t_0)\subset K_{R}^\lambda(z_0)$, we have
$$
\begin{aligned}
\fint_{K_{R/4}^\lambda(z_0)}|D(u-v)|^{p-1}dxdt\lesi& \fint_{K_{R/4}^\lambda(z_0)}|D(u-w)|^{p-1}dxdt+\fint_{K_{R/4}^\lambda(z_0)}|D(w-h)|^{p-1}dxdt\\
&+\fint_{K_{R/4}^\lambda(z_0)}|D(h-v)|^{p-1}dxdt\\
\lesi& \fint_{K_{R/4}^\lambda(z_0)}|D(u-w)|^{p-1}dxdt+\fint_{K_{R/2}^\lambda(0,t_0)}|D(w-h)|^{p-1}dxdt\\
&+\fint_{K_{R/2}^\lambda(0,t_0)}|D(h-v)|^{p-1}dxdt.
\end{aligned}
$$
At this stage, applying \eqref{eq1b-prop3}, \eqref{eq-proof Prop 3 boundary} and \eqref{eq-lem2-boundary}, we get \eqref{eq-u-v-prop3}.

The estimate \eqref{eq-Lvc v-prop3-boundary} follows immediately from \eqref{eq1a-prop3} and the following 
$$
\|D\bar{v}\|_{L^\vc(Q_{3R/8}^\lambda(z_0))}\leq \|D\bar{v}\|_{L^\vc(Q_{R/4}^\lambda(0,t_0))} \ \ \text{(due to \eqref{eq2-new coordinate})}. 
$$

Hence, to complete the proof, we need only to prove \eqref{eq1a-prop3} and \eqref{eq1b-prop3}.

\noindent\underline{Proof of \eqref{eq1a-prop3} and \eqref{eq1b-prop3}:}
 
By using suitable scaled maps, it suffices to prove inequalities above for $\lambda=1$ and $R=8$, that is, if $h$ is a solution to \eqref{AppProb2s-boundary} with $\lambda=1, R=8$, then there exists a weak solution $v$ to the problem \eqref{AppProb3-boundary} with $\lambda=1, R=8$ such that 
\begin{equation}\label{eq1- normalized}
\|D\bar{v}\|_{L^\vc(Q_{2}(0,t_0))}\lesi 1,
\end{equation}
and
\begin{equation}\label{eq2- normalized}
\fint_{K_3(0,t_0)}|D(h-v)|^{p}dxdt\leq \epsilon^p.
\end{equation}

To do this, we first prove that 
\begin{equation}\label{eq-lvc v}
\|D\bar{v}\|_{L^\vc(Q_{2}(0,t_0))}\lesi 1,
\end{equation}
and
\begin{equation}\label{eq-h-v}
\fint_{Q^+_4(0,t_0)}|h-v|^{p}dxdt\leq \epsilon^p.
\end{equation}

Indeed, we assume, to the contrary, that there exist an $\epsilon>0$, a sequence of domains $\{\Omega_k\}$ such that
 \begin{equation}\label{geometricconditionOmegak}
 B_4^+\subset \Omega^k_4\subset \left\{x\in B_4: x_n>-\f{16}{k}\right\},
 \end{equation}
 and a  sequence of functions $\{h^k\}$ which solves the problem 
 \begin{equation}\label{AppProb2sk-boundary}
 \left\{
 \begin{aligned}
 &h^k_t-\di\, \overline{\mathbf{a}}_{B_4}(D h^k,t)=0 \quad &\text{in}& \quad K^k_{4}(0,t_0):=(\Om^k\cap B_4)\times (t_0-4^2, t_0+4^2)\\
 &h^k=0 \quad &\text{on}& \quad \partial_w K^k_{4}(0,t_0).
 \end{aligned}\right.
 \end{equation}
satisfying
 \begin{equation}
 \label{eq h-ApproxPro1 Ok}
 \fint_{K^k_4(0,t_0)}|D h^k|^p\lesi 1.
 \end{equation}
 But, we have
 \begin{equation}
 \label{eq2 v-ApproxPro1 Ok}
 \fint_{Q_4^+(0,t_0)}|h^k-v|^p> \epsilon,
 \end{equation}
 for any weak solution $v$ to the problem \eqref{AppProb3-boundary} with
 \begin{equation}
 \label{eq1 vk-ApproxPro1}
 \fint_{Q_4^+(0,t_0)}|Dv|^p\lesi 1.
 \end{equation} 	
 From \eqref{geometricconditionOmegak}, \eqref{eq h-ApproxPro1 Ok}, \eqref{eq1-functiona} and Poincar\'e's inequality, we have
 $$
 \fint_{Q_4^+(0,t_0)}|Dh^k|^pdxdt \leq \fint_{K^k_4(0,t_0)}|Dh^k|^pdxdt\leq \fint_{K^k_4(0,t_0)}|Dh^k|^pdxdt\lesi 1, 
 $$
 and
 $$
 \begin{aligned}
 \|h^k_t\|_{L^{p'}(t_0-4^2, t_0+4^2; W^{-1,p'}(B_4^+))}&=\|\di\, \overline{\mathbf{a}}_{B_4}(D h^k,t)\|_{L^{p'}(t_0-4^2, t_0+4^2; W^{-1,p'}(B_4^+))}\\
 	&\leq \|\overline{\mathbf{a}}_{B_4}(D h^k,t)\|_{L^{p'}(t_0-4^2, t_0+4^2; L^{p'}(B_4^+))}\\
 	&\leq \|(Dh^k)^{p-1}\|_{L^{p'}(t_0-4^2, t_0+4^2; L^{p'}(B_4^+))}\\
 	&\lesi  \Big(\int_{K^k_4(0,t_0)}|D h^k|^p\Big)^{\f{p-1}{p}}\lesi 1.
 \end{aligned}		
 $$
 Therefore, by Aubin-Lions Lemma in \cite[Chapter 3]{Sh}, there exists $h^0$ with $h^0\in L^p(t_0-4^2,t_0+4^2; W^{1,p}(B_4^+))$ and  $h^0_t\in L^{p'}(t_0-4^2,t_0+4^2; W^{-1,p'}(B_4^+))$ such that there exists a subsequence of $\{h^k\}$, which is still denoted by $\{h^k\}$, satisfying
 $$
 h^k\to h^0, \ \ \text{strongly in $L^p(t_0-4^2,t_0+4^2; L^p(B_4^+))$},
 $$
 $$
 Dh^k\to Dh^0, \ \ \text{weakly in $L^p(t_0-4^2,t_0+4^2; L^p(B_4^+))$},
 $$
and
 $$
 h_t^k\to h_t^0, \ \ \text{weakly in  $L^{p'}(t_0-4^2,t_0+4^2; W^{-1,p'}(B_4^+))$}.
 $$
 As a direct consequence, we have
 $$
 \int_{Q_4^+(0,t_0)}|Dh^0|^pdxdt\lesi \liminf_{k} \int_{Q_4^+(0,t_0)}|Dh^k|^pdxdt\lesi 1.
 $$
From \eqref{geometricconditionOmegak}, we have 
$$
h^0=0 \quad \text{on} \quad Q_4\cap \{x: x_n=0\}\times (t_0-4^2, t_0+4^2).
$$
Therefore, $h^0$ solves
 $$
 \left\{
 \begin{aligned}
 &h^0_t-\di \,\overline{\mathbf{a}}_{B_4}(D h^0,t)=0 \quad &\text{in}& \quad Q_4^+(0,t_0),\\
 &h^0=0 \quad &\text{on}& \quad Q_4\cap \{x: x_n=0\}\times (t_0-4^2, t_0+4^2).
 \end{aligned}\right.
 $$
  This contradicts to \eqref{eq2 v-ApproxPro1 Ok} by taking $v=h_0$ and $k$ sufficiently large. Hence, \eqref{eq-lvc v} and \eqref{eq-h-v} are proved.
  
  \medskip
  
  We now turn to prove \eqref{eq1- normalized} and \eqref{eq2- normalized}.	Let $\bar{v}$ be a zero extension of $v$ to $Q_4(z_0)$. Then it can be verified that $\bar{v}$ solves 
	$$
	\bar{v}_t-\di\, \overline{\mathbf{a}}_{B_4}(D \bar{v},t)=D_{x_n}\left[\overline{\mathbf{a}}^n_{B_4}(D\bar{v}(x',0,t))\chi_{\{x:x_n<0\}}\right] \ \ \text{in $Q_4(0,t_0)$},
	$$
	where $x=(x',x_n)$ and $\mathbf{a}=(\mathbf{a}^1,\ldots,\mathbf{a}^n)$.
	
	Therefore, $h-\bar{v}$ solve
	$$
	(h-\bar{v})_t-\di\, \overline{\mathbf{a}}_{B_4}(D (h-\bar{v}),t)= -D_{x_n}\left[\overline{\mathbf{a}}^n_{B_4}(D\bar{v}(x',0,t))\chi_{\{x:x_n<0\}}\right]\quad \text{in} \quad K_4(0,t_0).
	$$
	By a standard argument as in the proof of Lemma \ref{Caccioppoli-boundary}, we can show that
	\begin{equation}\label{eq1-proof-prop3.15}
		\begin{aligned}
	\fint_{K_3(0,t_0)}&|D(h-\bar{v})|^pdxdt\\
	&\lesi\fint_{K_4(0,t_0)}|h-\bar{v}|^pdxdt+\fint_{K_4(0,t_0)}|h-\bar{v}|^2dxdt+\fint_{K_4(0,t_0)}|D\bar{v}(x',0,t)\chi_{\{x:x_n<0\}}|^pdxdt.
	\end{aligned}
	\end{equation}
	Using \eqref{eq-h-v}, we discover that 
	\begin{equation}\label{eq2-proof-prop3.15}
	\int_{K_3(0,t_0)}|h-\bar{v}|^pdxdt\leq C\int_{Q^+_4(0,t_0)}|h-\bar{v}|^pdxdt+\int_{K_4(0,t_0)\backslash Q^+_4(0,t_0)}|h|^pdxdt\leq C(\epsilon_1 +\mathcal{O}(\delta)).
	\end{equation}
	It is not difficult to see that 
	\begin{equation}\label{eq3-proof-prop3.15}
	\fint_{K_4(0,t_0)}|h-\bar{v}|^2dxdt\leq C(\epsilon_1 +\mathcal{O}(\delta)).
	\end{equation}
	Moreover, by \eqref{eq1-new coordinate}, we have
	\begin{equation}\label{eq4-proof-prop3.15}
	\begin{aligned}
	\fint_{K_3(0,t_0)}|D\bar{v}(x',0,t)\chi_{\{x:x_n<0\}}|^pdxdt&\leq \fint_{K_3(0,t_0)\cap \{x: -12\delta<x_n\leq 0\}\times (t_0-3^3,t_0)}|D\bar{v}(x',0,t)|^pdxdt\\
	&\leq \mathcal{O}(\delta).
	\end{aligned}
	\end{equation}
	Taking the estimates \eqref{eq1-proof-prop3.15}, \eqref{eq2-proof-prop3.15}, \eqref{eq3-proof-prop3.15} and \eqref{eq4-proof-prop3.15} into account, we imply \eqref{eq2- normalized}.
	
	The assertion \eqref{eq1- normalized} follows immediately from \eqref{eq2- normalized}:
	$$
	\begin{aligned}
	\|v\|^p_{L^\vc(Q_2^+(0,t_0))}\leq& C\fint_{Q_4^+(0,t_0)}|Dv|^p\\
	\lesi&\fint_{Q_4^+(0,t_0)} |D(h-\bar{v})|^pdz+\fint_{Q_4^+(0,t_0)}|D(h-w)|^pdz+\fint_{Q_{4}^+(0,t))} |Dw^pdz\\
	\lesi&\fint_{Q_4^+(0,t_0)} |D(h-\bar{v})|^pdz+\fint_{Q_{4}^+(0,t))}|D(h-w)|^pdz+\fint_{K_{8}^\lambda(z_0))} |Dw|^pdz\\
	\lesi& 1,
	\end{aligned}
	$$
	where in the first inequality we used the H\"older estimate of $v$ near  the flat boundary in \cite{L}.
	
	This completes our proof.
	
\end{proof}
    	
\section{The global Marcinkiewicz estimates}

This section is devoted to the proof of Theorem \ref{mainthm1}.

\medskip

Let $\mu\in L^{1,\theta}(\Om_T)$ with $ 1<\theta\leq n+2$ and $u$ be a SOLA to \eqref{ParabolicProblem}. We assume that $0<\delta<\f{1}{50}$. Fix $1\leq s_1<s_2\leq 2$, $R<\min\{R_0/4, 1/4\}$ and $z_0\in \Om_T$.  We set
\begin{equation}\label{lambda0}
\lambda_0:=\fint_{K_{2R}(z_0)}|Du|^{p-1}dz +\left[\f{1}{\delta}\f{|\mu|(K_{2R}(z_0))}{|K_{2R}(z_0)|}\right]^{\f{p-1}{m}}+1,
\end{equation}
where $m=p-1+\f{1}{\theta-1}$.

For $\lambda>0$, we now define the level set
$$
E_{s_1}(\lambda)=\{z\in K_{s_1R}(z_0): |Du(z)|>\lambda\}.
$$

For $z\in E_{s_1}(\lambda)$, we define $$
G_z(r)= \fint_{K_{r}^\lambda(z)}|Du|^{p-1}dz +\left[\f{1}{\delta}\f{|\mu|(K^\lambda_{r}(z))}{|K^\lambda_{r}(z)|}\right]^{\f{p-1}{m}}.
$$
By Lebesgue's differentiation theorem, we have
\begin{equation}\label{eq Gz0}
\lim_{r\to 0}G_z(r)=|Du(z)|^{p-1}>\lambda^{p-1}.
\end{equation}
Note that for $\f{(s_2-s_1)R}{10^5}<r\leq (s_2-s_1)R$,  $z\in E_{s_1}(\lambda)$ and $\lambda>1$, we have $K_r^\lambda(z)\subset K_{2R}(z_0)$. Hence, for a such $r$ one gets that
\begin{equation}\label{eq1-Gz}
\begin{aligned}
G_z(r)&=\fint_{K^\lambda_{r}(z)}|Du|^{p-1}dz +\left[\f{1}{\delta}\f{|\mu|(K^\lambda_{r}(z))}{|K^\lambda_{r}(z)|}\right]^{\f{p-1}{m}}\\	
&\leq \f{|K_{2R}(z_0)|}{|K_{r}^\lambda(z)|}\fint_{K_{2R}(z_0)}|Du|^{p-1}dz + \left[\f{1}{\delta}\f{|\mu|(K_{2R}(z_0))}{|K^\lambda_{r}(z)|}\right]^{\f{p-1}{m}}\\
&\leq \f{|K_{2R}(z_0)|}{|K_{r}^\lambda(z)|}\fint_{K_{2R}(z_0)}|Du|^{p-1}dz + \left[\f{|K_{2R}(z_0)|}{|K^\lambda_{r}(z)|}\right]^{\f{p-1}{m}}\left[\f{1}{\delta}\f{|\mu|(K_{2R}(z_0))}{|K^\lambda_{2R}(z_0)|}\right]^{\f{p-1}{m}}\\
&\leq \f{|K_{2R}(z_0)|}{|K_{r}^\lambda(z)|}\left\{\fint_{K_{2R}(z_0)}|Du|^{p-1}dz + \left[\f{1}{\delta}\f{|\mu|(K_{2R}(z_0))}{|K^\lambda_{2R}(z_0)|}\right]^{\f{p-1}{m}}\right\}\\
&\leq \f{|K_{2R}(z_0)|}{|K_{r}^\lambda(z)|}\lambda_0\\
&\leq \f{|Q_{2R}(z_0)|}{|Q_{r}^\lambda(z)|}\f{|Q_{r}^\lambda(z)|}{|K_{r}^\lambda(z)|}\lambda_0\\
&\leq 4^n\f{(2R)^{n+2}}{\lambda^{2-p}r^{n+2}}\lambda_0.
\end{aligned}
\end{equation}

We now fix
\begin{equation}\label{C0}
\lambda>4^n\Big(\f{2\times 10^5}{s_2-s_1}\Big)^{n+2}\lambda_0=\tilde{C}_0 \lambda_0.
\end{equation}
Then from \eqref{eq1-Gz}, we obtain
$$
G_z(r)<\lambda^{p-1}, \ \ \ \text{for all $r\in [10^{-5}(s_2-s_1)R,(s_2-s_1)R]$}.
$$
This together with \eqref{eq Gz0} implies that for each $z\in E_{s_1}(\lambda)$ there exists $0<r_z<10^{-5}(s_2-s_1)R$ so that 
$$
G_z(r_z)=\lambda^{p-1}, \ \ \ \text{and $G_z(r)<\lambda^{p-1}$ for all $r\in (r_z, (s_2-s_1)R)$}.
$$
We now apply Vitali's covering lemma to obtain the following result directly.

\begin{lem}
	\label{coveringlemma}
	There exists a countable disjoint family $\{K_{r_i}^\lambda(z_i)\}_{i\in \mathcal{I}}$ with  $r_i<\f{(s_2-s_1)R}{10^5}$ and $z_i=(x_i,t_i)\in E_{s_1}(\lambda)$ such that:
\begin{enumerate}[{\rm (a)}]
	\item $E_{s_1}(\lambda)\subset \bigcup_{i}K_{5r_i}^\lambda(z_i)$;
	\item $G_{z_i}(r_i)=\lambda^{p-1}$,  and $G_{z_i}(r)<\lambda^{p-1}$ for all $r\in (r_i, (s_2-s_1)R)$.
\end{enumerate}
\end{lem}

For each $i$,	from Lemma \ref{coveringlemma} we have
	$$
	\begin{aligned}
	\lambda^{p-1}&=\fint_{K^\lambda_{r_i}(z_i)}|Du|^{p-1}dz +\left[\f{1}{\delta}\f{|\mu|(K^\lambda_{r_i}(z_i))}{|K^\lambda_{r_i}(z_i)|}\right]^{\f{p-1}{m}}.
	\end{aligned}
	$$
	This implies that
	$$
	\f{\lambda^{p-1}}{2}\leq  \fint_{K_{r_i}^\lambda(z_i)}|Du|^{p-1}dxdt \ \ \ \ \text{or} \ \ \ \ \f{\lambda^{p-1}}{2}\leq  \left[\f{1}{\delta}\f{|\mu|(K^\lambda_{r_i}(z_i))}{|K^\lambda_{r_i}(z_i)|}\right]^{\f{p-1}{m}}.
	$$
	This is equivalent to
	\begin{equation}\label{eq-I}
	K_{r_i}^\lambda(z_i)\leq \f{2}{\lambda^{p-1}} \int_{K_{r_i}^\lambda(z_i)}|Du|^{p-1}dxdt,
	\end{equation}
	or
	\begin{equation}\label{eq-J}
	\lambda^m |K^\lambda_{r_i}(z_i)|\leq\f{2^{\f{m}{p-1}}}{\delta}|\mu|(K^\lambda_{r_i}(z_i)).
	\end{equation}
	We now set 
	$$
	\mathcal{M}=\{i: \ \ \text{\eqref{eq-I} holds true}\}, \ \ \ \mathcal{N}=\{i: \ \ \text{\eqref{eq-J} holds true}\}.
	$$
	Then, $\mathcal{I}=\mathcal{M}\cup \mathcal{N}$.
	
	We have the following estimate.
	\begin{prop}
		\label{prop1}
		For each $i\in \mathcal{M}$ we have
		\begin{equation}\label{eq1-wKilambda}
		\begin{aligned}
		|K_{r_i}^\lambda(z_i)|
		&\lesi |K_{r_i}^\lambda(z_i)\cap E_{s_2}(\lambda/4)|.
		\end{aligned}
		\end{equation}
	\end{prop}
	\begin{proof}
		Let $u_k$ be a weak solution to the problem \eqref{RegularizedParabolicProblem} for each $k\in \mathbb{N}$. Since $u_k\to u$ in $L^{p-1}(0, T; W^{1,p-1}_0(\Om))$, from \eqref{eq-I} there exists $k_1$ such that for all $k\geq k_1$,
		$$
		K_{r_i}^\lambda(z_i)\leq \f{3}{\lambda^{p-1}} \int_{K_{r_i}^\lambda(z_i)}|Du_k|^{p-1}dxdt.
		$$
		For each $k\in \mathbb{N}$ and $s>0$, we define $E_{k,s}(\lambda)=\{z\in K_{sR}(z_0): |Du_k(z)|>\lambda\}$. Due to $K_{r_i}^\lambda(z_i)\subset K_{s_2R}(z_0)$, we have
	$$
	\begin{aligned}
	|K_{r_i}^\lambda(z_i)|	&\leq \f{3}{\lambda^{p-1}}\int_{K_{r_i}^\lambda(z_i)\backslash E_{k,s_2}(\lambda/4)}|Du_k|^{p-1}dxdt+\f{3}{\lambda^{p-1}}\int_{K_{r_i}^\lambda(z_i)\cap E_{k,s_2}(\lambda/4)}|Du_k|^{p-1}dxdt\\
	& \leq \f{|K_{r_i}^\lambda(z_i)|}{4^{p-2}}+ \f{3}{\lambda^{p-1}}\int_{K_{r_i}^\lambda(z_i)\cap E_{k,s_2}(\lambda/4)}|Du_k|^{p-1}dxdt.
	\end{aligned}
	$$
	This implies 
	\begin{equation}\label{eq-Kwi}
	|K_{r_i}^\lambda(z_i)|\lesi \f{1}{\lambda^{p-1}}\int_{K_{r_i}^\lambda(z_i)\cap E_{k,s_2}(\lambda/4)}|Du_k|^{p-1}dxdt.
	\end{equation}
	Note that from the definitions of $r_i$, the index set $\mathcal{M}$ and the fact that $u_k\to u$ in $L^{p-1}(0, T; W^{1,p-1}_0(\Om))$, there exists $k_2$ such that for all $k\geq k_2$ we have
	$$
	\f{\lambda^{p-1}}{3}\leq \fint_{K_{r_i}^\lambda(z_i)}|Du_k|^{p-1}dxdt, \ \ \ \fint_{K_{4r_i}^\lambda(z_i)}|Du_k|^{p-1}dxdt< 3\lambda^{p-1},
	$$
	and
	$$
	\f{|\mu_k|(K_{4r_i}^\lambda(z_i))}{|K_{4r_i}^\lambda(z_i)|}\leq \delta\lambda^m.
	$$
		
	By Holder's inequality, for a fixed $\nu\in (p-1, p-1+\f{1}{n+1})$ we have
	\begin{equation}\label{eq1-Du}
	\begin{aligned}
	\Big(\f{1}{|K_{r_i}^\lambda(z_i)|}\int_{K_{r_i}^\lambda(z_i)\cap E_{k,s_2}(\lambda/4)}&|Du_k|^{p-1}dxdt\Big)^{\f{1}{p-1}}\\
	&\leq \Big(\f{1}{|K_{r_i}^\lambda|}\int_{K_{r_i}^\lambda(z_i)\cap E_{k,s_2}(\lambda/4)}|Du_k|^{\nu}dxdt\Big)^{\f{1}{\nu}}
	\Big(\f{|K_{r_i}^\lambda(z_i)\cap E_{k,s_2}(\lambda/4)|}{|K_{r_i}^\lambda(z_i)|}\Big)^{1-\f{1}{\nu}},
	\end{aligned}
	\end{equation}
	Since $u$ is not a weak solution, we can not apply the estimates results in Section 3 and Section 4 directly. However, we can apply the estimates results to estimate $u$ via an approximation scheme.
	 
	For each $k$ and $i$, consider the following equation
	$$\left\{
	\begin{aligned}
		&(w_k^i)_t-\di\, \mathbf{a}(D w_k^i,x,t)=0 \quad &\text{in}& \quad Q^\lambda_{4r_i},\\
		&w_k^i=u_k \quad &\text{on}& \quad \partial_p Q^\lambda_{4r_i}.
	\end{aligned}\right.
	$$
	At this stage, arguing similarly to \eqref{eq-proof Prop 3 boundary}, we have
	\begin{equation}\label{eq-unu}
	\fint_{K_{4r_i}^\lambda(z_i)}|D(u_k-w_k^i)|^{\nu}dxdt\leq \mathcal{O}(\delta)\lambda^{\nu}.
	\end{equation}
	On the other hand, the argument used in the proof of \eqref{eq- wp} also implies that 
	$$
	c^{-1}\lambda^{p}\leq \fint_{K^\lambda_{r_i}(z_i)}|Dw_k^i|^{p}dxdt, \ \ \ \fint_{K^\lambda_{2r_i}(z_i)}|Dw_k^i|^{p}dxdt\leq c\lambda^{p},
	$$
	for some $c\geq 1$.
	
	As a consequence,
	$$
	\fint_{K^\lambda_{r_i}(z_i)}|Dw_k^i|^{\nu}dxdt\leq c\lambda^{\nu}.
	$$
	This along with \eqref{eq-unu} yields
	$$
	\fint_{K^\lambda_{r_i}(z_i)}|Du_k|^{\nu}dxdt\lesi \lambda^{\nu}
	$$
	provided that $\delta$ is sufficiently small.
	
	Inserting this into \eqref{eq1-Du}, we get that
	$$
	\begin{aligned}
	\Big(\f{1}{|K_{r_i}^\lambda(z_i)|}\int_{K_{r_i}^\lambda(z_i)\cap E_{k,s_2}(\lambda/4)}|Du_k|^{p-1}dxdt\Big)^{\f{1}{p-1}}
	&\leq c\lambda
	\Big(\f{|K_{r_i}^\lambda(z_i)\cap E_{k,s_2}(\lambda/4)|}{|K_{r_i}^\lambda(z_i)|}\Big)^{1-\f{1}{\nu}},
	\end{aligned}
	$$
	or equivalently,
		\begin{equation}\label{eq2-Du}
		\begin{aligned}
		\int_{K_{r_i}^\lambda(z_i)\cap E_{k,s_2}(\lambda/4)}|Du_k|^{p-1}dxdt
		&\leq c\lambda^{p-1}|K_{r_i}^\lambda(z_i)|
		\Big(\f{|K_{r_i}^\lambda(z_i)\cap E_{k,s_2}(\lambda/4)|}{|K_{r_i}^\lambda(z_i)|}\Big)^{\f{(p-1)(\nu-1)}{\nu}}.
		\end{aligned}
		\end{equation}
		This, in combination with \eqref{eq-Kwi}, gives that
		$$
		|K_{r_i}^\lambda(z)|\leq c|K_{r_i}^\lambda(z_i)|
		\Big(\f{|K_{r_i}^\lambda(z_i)\cap E_{k,s_2}(\lambda/4)|}{|K_{r_i}^\lambda(z_i)|}\Big)^{\f{(p-1)(\nu-1)}{\nu}}.
		$$
Therefore
		$$
		|K_{r_i}^\lambda(z_i)|
		\lesi |K_{r_i}^\lambda(z_i)\cap E_{k,s_2}(\lambda/4)|.
		$$
		Letting $k\to \vc$, we get 	\eqref{eq1-wKilambda} immediately.
	\end{proof}

\begin{prop}
	\label{prop2}
	There exists $N_0>1$ so that for any $\lambda>\tilde{C}_0\lambda_0$ we have 
	\begin{equation}\label{eq-ENlambda}
	\begin{aligned}
	\sum_{i\in \mathcal{M}}|E_{s_1}(N_0\lambda)\cap K^\lambda_{5r_i}(z_i)|\leq &\epsilon E_{s_2}(\lambda/4).
	\end{aligned}
	\end{equation}
	As a consequence, we have
	\begin{equation}\label{eq2-ENlambda}
	|E_{s_1}(N_0\lambda)|\leq \epsilon E_{s_2}(\lambda/4)+c\lambda^{-m}|\mu|(K_{s_2R}(z_0)).
	\end{equation}
\end{prop}
\begin{proof}
We now set
$$
\mathcal{M}_1:=\{i: B_{40r_i}^\lambda(x_i)\subset \Om\}, \ \text{and} \ \  \mathcal{M}_2:=\{i: B_{40r_i}^\lambda(z_i)\cap \Om^c\neq \emptyset\}.
$$	
For $i\in \mathcal{M}_1$, from the definition of $\mathcal{M}_1$ and Lemma \ref{coveringlemma}, we have
$$
\lambda^{p-1}\lesi \fint_{Q_{10r_i}^\lambda(z_i)}|Du|^{p-1}dxdt, \ \ \ \fint_{Q_{40r_i}^\lambda(z_i)}|Du|^{p-1}dxdt< \lambda^{p-1},
$$
and
$$
\f{|\mu|(Q_{40r_i}^\lambda(z_i))}{|Q_{40r_i}^\lambda(z_i)|}\leq \delta\lambda^m.
$$
Let $\{u_k\}$ be weak solutions to the problems \eqref{RegularizedParabolicProblem} for each $k\in \mathbb{N}$. Then from these two estimates above there exists $k_1>0$ so that for all $k\geq k_1$ we have
$$
\lambda^{p-1}\lesi \fint_{Q_{10r_i}^\lambda(z_i)}|Du_k|^{p-1}dxdt, \ \ \ \fint_{Q_{40r_i}^\lambda(z_i)}|Du_k|^{p-1}dxdt< \lambda^{p-1},
$$
and
$$
\f{|\mu_k|(Q_{40r_i}^\lambda(z_i))}{|Q_{40r_i}^\lambda(z_i)|}\leq \delta\lambda^m.
$$

Then applying Proposition \ref{appProp3-inter}, for each $k\geq k_1$ and $i\in \mathcal{M}_1$ we can find $v_{k}^i$ such that
\begin{equation}\label{eq-approx M1}
\|Dv_{k}^i\|_{L^\vc(Q^\lambda_{5r_i}(z_i))}\leq A_1\lambda, \ \ \fint_{Q^\lambda_{10r_i}(z_i)}|D(u_k-v_{k}^i)|^{p-1}\leq (\epsilon\lambda)^{p-1}.
\end{equation}
\medskip

For $i\in \mathcal{M}_2$, pick $\bar{x}_i\in B_{10r_i}(x_i)\cap \partial \Omega$. Setting  $\bar{z}_i=(\bar{x}_i,t_i)$, then we have
\begin{equation}
\label{eq-change the ball}
Q_{5r_i}^\lambda(z_i)\subset Q_{15r_i}^\lambda(\bar{z}_i)\subset Q_{280r_i}^\lambda(\bar{z}_i)\subset Q_{500r_i}^\lambda(z_i).
\end{equation}
Therefore, from the definition of $\mathcal{M}_2$ and Lemma \ref{coveringlemma}, we have
$$
\lambda^{p-1}\lesi \fint_{K_{120r_i}^\lambda(\bar{z}_i)}|Du|^{p-1}dxdt, \ \ \ \fint_{K_{480r_i}^\lambda(\bar{z}_i)}|Du|^{p-1}dxdt\lesi \lambda^{p-1},
$$
and
$$
\f{|\mu|(K_{480r_i}^\lambda(\bar{z}_i))}{|K_{480r_i}^\lambda(\bar{z}_i)|}\leq \delta\lambda^m.
$$
Hence, there exists $k_2$ such that for all $k\geq k_2$ we have
$$
\lambda^{p-1}\lesi \fint_{K_{120r_i}^\lambda(\bar{z}_i)}|Du_k|^{p-1}dxdt, \ \ \ \fint_{K_{480r_i}^\lambda(\bar{z}_i)}|Du_k|^{p-1}dxdt\lesi \lambda^{p-1},
$$
and
$$
\f{|\mu_k|(K_{480r_i}^\lambda(\bar{z}_i))}{|K_{480r_i}^\lambda(\bar{z}_i)|}\leq \delta\lambda^m.
$$

We now apply Proposition \ref{approxPropBoundary-1} to find a function $v_k^i$, or each $k\geq k_2$ and $i\in \mathcal{M}_2$ so that 
\begin{equation*}
\|Dv_k^i\|_{L^\vc(Q^\lambda_{15r_i}(\bar{z}_i))}\leq c\lambda, \ \ \fint_{K^\lambda_{30r_i}(\bar{z}_i)}|D(u_k-v_k^i)|^{p-1}\leq (\epsilon\lambda)^{p-1}.
\end{equation*}
This together with \eqref{eq-change the ball} implies
\begin{equation}\label{eq-approx M2}
\|Dv_k^i\|_{L^\vc(Q^\lambda_{5r_i}(z_i))}\leq A_2\lambda, \ \ \fint_{K^\lambda_{10r_i}(z_i)}|D(u_k-v_k^i)|^{p-1}\leq (\epsilon\lambda)^{p-1}.
\end{equation}

Taking $N_0=\max\{2A_1,2A_2\}$, from \eqref{eq-approx M1} and \eqref{eq-approx M2} we have, for $k\geq \max\{k_1,k_2\}$,
$$
\begin{aligned}
\sum_{i\in \mathcal{M}}&|E_{s_1}(N_0\lambda)\cap K^\lambda_{5r_i}(z_i)|\\
&\leq \sum_{i\in \mathcal{M}}|\{z\in K_{5r_i}^\lambda(z_i): |Du_k(z)|>N_0\lambda/2\}|+ \sum_{i\in \mathcal{M}}|\{z\in K_{5r_i}^\lambda(z_i): |D(u-u_k)(z)|>N_0\lambda/2\}|\\
&\lesi \sum_{j=1}^2\sum_{i\in \mathcal{M}_j}\Big[|\{z\in K_{5r_i}^\lambda(z_i): |D(u_k-v_k^i)(z)|>N_0\lambda\}|+|\{z\in K_{5r_i}^\lambda(z_i): |Dv_k^i(z)|>N_0\lambda/2\}|\Big]\\
& \ \ \ \ \ +|\{z\in \Om_T: |D(u-u_k)(z)|>N_0\lambda/2\}|\\
&\lesi \sum_{j=1}^2\sum_{i\in \mathcal{M}_j}|\{z\in K_{5r_i}^\lambda(z_i): |D(u_k-v_k^i)(z)|>N_0\lambda\}|++|\{z\in \Om_T: |D(u-u_k)(z)|>N_0\lambda/2\}|\\
&\lesi \sum_{j=1}^2\sum_{i\in \mathcal{M}_j} \f{1}{(N_0\lambda)^{p-1}}\int_{K_{5r_i}^\lambda(z_i)}|D(u_k-v_k^i)|^{p-1}dz+\f{1}{(N_0\lambda)^{p-1}}\int_{\Om_T}|D(u_k-u)|^{p-1}dz\\
&\lesi \epsilon |K_{5r_i}^\lambda(z_i)|++\f{1}{(N_0\lambda)^{p-1}}\int_{\Om_T}|D(u_k-u)|^{p-1}dz\\
&\lesi \epsilon |K_{r_i}^\lambda(z_i)|++\f{1}{(N_0\lambda)^{p-1}}\int_{\Om_T}|D(u_k-u)|^{p-1}dz.
\end{aligned}
$$
Letting $k\to \vc$ and using the fact that $K_{r_i}^\lambda(z_i)\subset K_{s_2R}(z_0)$, the estimate \eqref{eq-ENlambda} follows as desired.

To prove \eqref{eq2-ENlambda}, we observe that from Lemma \ref{coveringlemma}, \eqref{eq-ENlambda} and the fact that $\mathcal{I}=\mathcal{M}\cup \mathcal{N}$, we have
$$
\begin{aligned}
|E_{s_1}(N_0\lambda)|&\leq \sum_{i\in \mathcal{M}}|E_{s_1}(N_0\lambda)\cap K^\lambda_{5r_i}(z_i)|+\sum_{i\in \mathcal{N}}|E_{s_1}(N_0\lambda)\cap K^\lambda_{5r_i}(z_i)|\\
&\leq \epsilon E_{s_2}(\lambda/4)+\sum_{i\in \mathcal{N}}| K^\lambda_{5r_i}(z_i)|.
\end{aligned}
$$
From the definition of $\mathcal{N}$ and the fact that $K^\lambda_{r_i}(z_i)\subset K_{s_2R}(z_0)$, we have
$$
\sum_{i\in \mathcal{N}}|K^\lambda_{5r_i}(z_i)|\leq C\sum_{i\in \mathcal{N}}|K^\lambda_{r_i}(z_i)|\leq C\lambda^{-m}\sum_{i\in \mathcal{N}}|\mu|(K^\lambda_{r_i}(z_i))\leq C\lambda^{-m}|\mu|(K_{s_2R}(z_0)),
$$
where in the last inequality we used the fact that $\{K^\lambda_{r_i}(z_i)\}$ is pairwise disjoint.
\end{proof}

We now recall the result in \cite[Lemma 4.3]{HL}.
\begin{lem}\label{HF'sLemma}
	Let $f$ be a bounded nonnegative function on $[a_1, a_2]$ with $0<a_1<a_2$. Assume that for any $a_1\leq x_1\leq x_2\leq a_2$ we have
	$$
	f(x_1)\leq \theta_1 f(x_2)+\f{A_1}{(x_2-x_1)^{\theta_2}}+A_2,
	$$
	where $A_1, A_2>0$, $0<\theta_1<1$ and $\theta_2>0$. Then, there exists $c=c(\theta_1,\theta_2)$ so that
	$$
	f(x_1)\leq c\Big[\f{A_1}{(x_2-x_1)^{\theta_2}}+A_2\Big].
	$$
\end{lem}

We now ready to give the proof of Theorem \ref{mainthm1}.

\begin{proof}[Proof of Theorem \ref{mainthm1}:]
	
	For each $k>0$ we define $|Du|_k=\min\{k, |Du|\}$. Then $|Du|_k\in \mathcal{M}^m(\Om_T)$ for all $k$. We set $E^k_s(\lambda)=\{z\in K_{sR}(z_0): |Du(z)|_k>\lambda\}$ for $s>0$.
	
	From \eqref{eq2-ENlambda}, it follows immediately that there exists $C$ independing of $k$ so that
	$$
	|E^k_{s_1}(N_0\lambda)|\leq \epsilon E^k_{s_2}(\lambda/4)+c\lambda^{-m}|\mu|(K_{s_2R}(z_0)), \ \ \  \lambda>\tilde{C}_0\lambda_0.
	$$
	Hence,
	$$
	\lambda^{m}|E^k_{s_1}(N_0\lambda)|\leq \epsilon \lambda^{m} E^k_{s_2}(\lambda/4)+C|\mu|(K_{s_2R}(z_0)), \ \ \ \ \lambda>\tilde{C}_0\lambda_0.
	$$
	This implies that 
	$$
	\begin{aligned}
	\sup_{\lambda>0}\lambda^{m}|E^k_{s_1}(N_0\lambda)|\leq &\sup_{0<\lambda\leq \tilde{C}_0\lambda_0 }\lambda^{m}|E^k_{s_1}(N_0\lambda)|+\sup_{\lambda> \tilde{C}_0\lambda_0 }\lambda^{m}|E^k_{s_1}(N_0\lambda)|\\
	&\leq (\tilde{C}_0\lambda_0)^m|K_{s_1R}(z_0)|+\epsilon \sup_{\lambda>0}\lambda^{m} |E^k_{s_2}(\lambda/4)|+C|\mu|(K_{s_2R}(z_0)).
	\end{aligned}
	$$
	Substituting the values of $\tilde{C}_0$ and $\lambda_0$ given by \eqref{lambda0} and \eqref{C0} into the inequality above, we obtain
	$$
	\begin{aligned}
	\||Du|_k\|^m_{\mathcal{M}^{m}(K_{s_1R}(z_0))}\leq& \epsilon \||Du|_k\|^m_{\mathcal{M}^{m}(K_{s_2R}(z_0))}+|\mu|(\Om_T)\\
	&+C\Big(\f{2\times 10^5}{s_2-s_1}\Big)^{(n+2)m}\Big[\fint_{K_{2R}(z_0)}|Du|^{p-1}dz +\left[\f{1}{\delta}\f{|\mu|(K_{2R}(z_0))}{|K_{2R}(z_0)|}\right]^{\f{p-1}{m}}|+1\Big]^m|\Om_T|.
	\end{aligned}
	$$
	Applying Lemma \ref{HF'sLemma}, we get that
	$$
	\begin{aligned}
	\||Du|_k\|^m_{\mathcal{M}^{m}(K_{s_1R}(z_0))}\lesi& |\mu|(\Om_T)+\Big(\f{2\times 10^5}{s_2-s_1}\Big)^{(n+2)m}\Big[\fint_{K_{2R}(z_0)}|Du|^{p-1}dz +\left[\f{1}{\delta}\f{|\mu|(K_{2R}(z_0))}{|K_{2R}(z_0)|}\right]^{\f{p-1}{m}}|+1\Big]^m|\Om_T|.
	\end{aligned}
	$$
	Taking $s_1=1, s_2=2$, we have
	$$
	\begin{aligned}
	\||Du|_k\|^m_{\mathcal{M}^{m}(K_{R}(z_0))}\lesi& |\mu|(\Om_T)+\Big[\fint_{K_{2R}(z_0)}|Du|^{p-1}dz +\left[\f{1}{\delta}\f{|\mu|(K_{2R}(z_0))}{|K_{2R}(z_0)|}\right]^{\f{p-1}{m}}|+1\Big]^m|\Om_T|\\
	\lesi& |\mu|(\Om_T)+\|Du\|_{L^{p-1}(\Om_T)}^{m(p-1)}+|\mu|(\Om_T)^{p-1}+1.\\
	\lesi& \Big[|\mu|(\Om_T)+\|Du\|_{L^{p-1}(\Om_T)}^{m}+1\Big]^{p-1}.
	\end{aligned}
	$$
	Since $\Om_T$ is bounded, we deduce that
	$$
	\||Du|_k\|^m_{\mathcal{M}^{m}(\Om_T)}\lesi \Big[|\mu|(\Om_T)+\|Du\|_{L^{p-1}(\Om_T)}^{m}+1\Big]^{p-1}.
	$$
	On the other hand, by tracking the constant in the proof of Lemma 2.2 in \cite{B.etal} we have
	$$
	\||Du|\|_{L^{p-1}(\Om_T)}\leq C|\mu|(\Om_T)^{\f{n+1}{n(p-1)}}.
	$$
	Hence,
	$$
	\||Du|_k\|_{\mathcal{M}^{m}(\Om_T)}\lesi \Big[|\mu|(\Om_T)^{\f{n+1}{n}}+|\mu|(\Om_T)^{\f{p-1}{m}}+1\Big]\lesi \Big[|\mu|(\Om_T)^{\f{n+1}{n}}+1\Big].
	$$
	Letting $k\to \vc$, we obtain
	$$
	\||Du|\|^m_{\mathcal{M}^{m}(\Om_T)}\lesi \Big[|\mu|(\Om_T)^{\f{n+1}{n}}+1\Big].
	$$
	This completes our proof.
	\end{proof}
	
\bigskip

\textbf{Acknowledgement.} The authors would like to thank the referee for useful comments and suggestions to improve the paper. The authors were supported by the research grant ARC DP140100649 from the Australian Research Council.

\end{document}